\documentclass[10pt]{article}
\usepackage[left=3.5cm, right=3.5cm]{geometry}

\usepackage[utf8]{inputenc}
\usepackage[T1]{fontenc}
\usepackage{amsmath}
\usepackage{amsfonts}
\usepackage{mathtools}
\usepackage{amssymb}
\usepackage[version=4]{mhchem}
\usepackage{stmaryrd}
\usepackage{bbold}
\usepackage{amsthm}
\usepackage{graphics} 
\usepackage{epsfig} 
\usepackage{thmtools}
\usepackage{algorithm}
\usepackage{algpseudocode}
\usepackage[section]{placeins}
\usepackage{xcolor}
\usepackage{subcaption}

\usepackage{enumitem}
\setlist[enumerate,1]{label=(\roman*),font=\itshape}

\declaretheoremstyle[
  headpunct={:},
  headindent = \parindent,
  bodyfont = \normalfont\itshape,
]{lemmastyle}
\declaretheorem[style=lemmastyle,name=Theorem]{theorem}
\declaretheoremstyle[
  headpunct={:},
  headindent = \parindent,
  bodyfont = \normalfont\itshape,
]{lemmastyle}
\declaretheorem[style=lemmastyle,name=Lemma]{lemma}
\declaretheoremstyle[
  headpunct={:},
  headindent = \parindent,
  notefont = \normalfont\itshape,
]{assumptionstyle}
\declaretheorem[style=assumptionstyle,name=Assumption]{assumption}
\declaretheoremstyle[
  headpunct={:},
  headindent = \parindent,
  notefont=\normalfont\itshape,
]{remarkstyle}
\declaretheorem[style=remarkstyle,name=Remark]{remark}

\DeclareMathOperator*{\argmin}{argmin}

\algnewcommand{\Initialize}[1]{%
  \State \textbf{Initialize:}
  \Statex \hspace*{\algorithmicindent}\parbox[t]{.8\linewidth}{\raggedright #1}
}

\title{ZO-JADE: Zeroth-order Curvature-Aware Multi-Agent Convex Optimization}
\author{%
Alessio Maritan, Luca Schenato\\
\normalsize University of Padova - Department of Information Engineering}
\date{}

\usepackage{hyperref}
\usepackage{fancyhdr}
\fancypagestyle{FirstPage}{
\lfoot{2475-1456 \copyright 2023 IEEE. Digital Object Identifier \href{https://ieeexplore.ieee.org/document/10139786}{10.1109/LCSYS.2023.3281745}} 
}

\begin{document}

\maketitle

\begin{abstract}
In this work we address the problem of convex optimization in a multi-agent setting where the objective is to minimize the mean of local cost functions whose derivatives are not available (e.g. black-box models). Moreover agents can only communicate with local neighbors according to a connected network topology. Zeroth-order (ZO) optimization has recently gained increasing attention in federated learning and multi-agent scenarios exploiting finite-difference approximations of the gradient using from $2$ (directional gradient) to $2d$ (central difference full gradient) evaluations of the cost functions, where $d$ is the dimension of the problem. The contribution of this work is to extend ZO distributed optimization by estimating the curvature of the local cost functions via finite-difference approximations. In particular, we propose a novel algorithm named ZO-JADE, that by adding just one extra point, i.e. $2d+1$ in total, allows to simultaneously estimate the gradient and the diagonal of the local Hessian, which are then combined via average tracking consensus to obtain an approximated Jacobi descent. Guarantees of semi-global exponential stability are established via separation of time-scales. Extensive numerical experiments on real-world data confirm the efficiency and superiority of our algorithm with respect to several other distributed zeroth-order methods available in the literature based on only gradient estimates.
\end{abstract}



\section{Introduction}

\thispagestyle{FirstPage}


Optimization is unquestionably the main pillar and enabler of several real-world applications, as it allows to get the best solution for any given problem,
improve the performance of processes and efficiently implement modern technologies.
However, most of the optimization techniques heavily rely on the knowledge of the derivatives of the objective function, which may be expensive or infeasible to obtain \cite{book_DFBO}. Such situations occur when the derivative is not available in analytical closed form, when automatic differentiation cannot be applied, in simulation-based optimization and when dealing with black-box models. Addressing these scenarios, the zeroth-order (ZO) and derivative-free optimization (DFO) branches have appeared in the literature, aiming to remove the dependency from the derivative information. 
Notable applications include adversarial attacks and explanations for deep neural networks, hyper-parameters tuning and policy search in reinforcement learning \cite{liu2020primer}.

The most basic and widely used approach to ZO optimization is to replace the true gradient with an estimator based on the Kiefer–Wolfowitz scheme \cite{kiefer1952stochastic}, which approximates the derivative by finite-differences of function values along a set of search directions. The DFO branch, instead, includes direct-search methods, interpolation models and trust-region frameworks \cite{book_IDFO}.

A common assumption in the ZO field is that evaluating the objective function is a costly process, for example involving complex simulations, and indeed the convergence rate is often expressed in terms of number of function calls. 
For this reason, in this work we extract as much information as possible from the query points, estimating the curvature information of the objective function and exploiting the latter to speed up convergence. We note that in the ZO literature most of the works limit themselves to extracting the gradient information, and only few of them take into account also the second-order derivative. Among them, \cite{ye2018hessian} uses a forward-difference gradient estimator in which the search directions are random Gaussian vectors scaled by the square root of the inverse Hessian matrix, and the latter is estimated using either a power-based algorithm or some heuristic methods. A similar idea is used in \cite{HIACRPBBP}, where two perturbations with inverse Hessian covariance are used to estimate both the gradient and the Hessian by central differencing. Other examples of second-order methods based on finite-difference approximations are \cite{kim2021curvature}, which applies Newton’s method along a single direction, and \cite{berahas2019derivative}, which proposes a ZO version of the well known limited memory BFGS. Finally, \cite{Balasubramanian} employs the second-order Stein’s identity to estimate the Hessian and presents a cubic-regularized Newton method. Regarding the DFO field, there are many algorithms involving quadratic models constructed by means of interpolation, Least Squares, Lagrange Polynomials \cite{conn2008geometry} or minimum Frobenius norm fitting \cite{powell2006newuoa}.

However, all the aforementioned works consider only the centralized setting.
In many relevant cases, such as distributed machine learning and estimation, robotic swarms and wireless sensor networks, it is required to spread the optimization process across a set of agents, which are often endowed with limited communication capabilities. Notably, the multi-agent setup does not require all the data and the processing power to be owned by a single centralized authority, increasing the scalability and the privacy guarantees of the algorithm. Currently, all the existing second-order solutions based on ZO estimates do not enjoy these appealing features. To address this shortcoming we propose a novel distributed second-order method for ZO optimization, which to the best of our knowledge is the first of its kind. Our algorithm, named ZO-JADE, takes advantage of the second-order derivative to obtain a Jacobi-like descent direction, i.e. the gradient is re-scaled by the inverse of the diagonal of the Hessian.
Using tools from Lyapunov theory and singular perturbation theory we formally prove the semi-global exponentially fast convergence of ZO-JADE to a point arbitrarily close to the minimum. The theoretical results are validated by numerical experiments, namely distributed ridge regression and one-vs-all classification on public datasets. These tests show that our algorithm, compared to several other zeroth-order multi-agent methods, converges much faster and reaches a possibly better solution.

\textit{Notation}: We denote with $I_d$ the $d$-dimensional identity matrix, with $e_k$ the $k$-th canonical vector, and with $\mathbb{0}$ and $\mathbb{1}$ the vectors whose components are all equal to 0 and 1, respectively. We indicate with $\odot$ and $\oslash$ the element-wise Hadamard product and division. The Euclidean or spectral norm is $\left\| \cdot \right\|$, and coincides with the largest singular value of the argument. We use the operator $\left[\cdot\right]_i$ to select the $i$-th element of the argument and diag$(\cdot)$ to select the diagonal of a matrix. The third order derivative tensor of $f$ evaluated at $x$ is denoted by $\left( D^3 f \right)(x)$.

\section{Problem Formulation}
In this paper we address the fully distributed optimization scenario, where $n$ agents can communicate with each other through bi-directional links according to a connected network topology.
Each agent $i$ has access to a local cost function $f_i: \mathbb{R}^d \to \mathbb{R}$, and the goal of the network is to cooperatively seek the optimal solution of the unconstrained minimization problem 
\begin{equation}
    f(x^\star) = \min_{x \in \mathbb{R}^d} \left\{ f(x) \coloneqq \frac{1}{n} \sum_{i=1}^n f_i(x) \right\}.
    \label{eq:problem_formulation}
\end{equation}

We consider the setting in which only zeroth-order information is available, namely it is only possible to evaluate the value of the local functions at the desired points, without access to gradients or higher-order derivatives.
To provide a rigorous analysis we characterize the objective function with the following assumption.

\begin{assumption} [Strong Convexity and Smoothness]
Let the global cost $f$ be three times continuously differentiable with Lipschitz continuous derivatives and $m$-strongly convex, i.e. there exist positive constants $L_1, L_2, L_3$ and $m$ such that
\begin{equation*}
\begin{aligned}
m I_d \leq \nabla^2 f (x) \leq L_1 I_d & \quad \forall x \in \mathbb{R}^d, \\
\left\| \nabla^2 f(x) - \nabla^2 f(y) \right\| \leq L_2 \left\|  x-y \right\| & \quad \forall x,y \in \mathbb{R}^d, \\
\left\|  \left( D^3 f \right)(x) - \left( D^3 f \right)(y) \right\| \leq L_3 \left\|  x-y \right\| & \quad \forall x,y \in \mathbb{R}^d.
\end{aligned}
\end{equation*}
\label{assumpt:convexity_smoothness}
\end{assumption}

We are interested in consensus-based algorithms, in which all the agents converge to the same optimal value of the variable to be optimized by just communicating with their neighbors. To do so, each agent updates its local variables using a weighted average of its neighbors' information, where the weights are provided by a mixing matrix $P$. The assumption below provides a formal definition of the distributed scenario under consideration.

\begin{assumption} [Communication network]
Let the communication network be represented by the time-invariant graph $\mathcal{G} = (\mathcal{N}, \mathcal{E})$ whose vertices $\mathcal{N} = \{1, \dots, n\}$ and edges $\mathcal{E}$ are the set of agents and the available communication links, respectively. Assume that the graph is undirected and connected, and that each node can only communicate with its single-hop neighbors. Let the network be associated with the consensus matrix $P \in \mathbb{R}^{n \times n}$, whose generic entry $p_{ij}$ is positive if edge $(i, j) \in \mathcal{E}$ and zero otherwise. Choose $P$ to be symmetric and doubly stochastic, i.e. such that $P \mathbb{1} = \mathbb{1}$ and $\mathbb{1}^T P = \mathbb{1}^T$. A matrix that satisfies these requirements can be constructed without knowledge of the entire network topology using e.g. the Metropolis-Hastings weights \cite{xiao2007distributed}.
\label{assumpt:graph}
\end{assumption}

\section{Zeroth-order Oracles}
Our algorithm requires to evaluate the first and second derivative of the local cost functions at the current point. Since in the zeroth-order setting these information are not available, we estimate them using central finite-difference schemes along orthogonal directions. Given a generic function $f$ and a small positive scalar $\mu$ we adopt the following approximations of the gradient and of the Hessian's diagonal:
\begin{equation}
\begin{aligned}
    \quad \hat{\nabla} f(\mu, x) &\coloneqq \sum_{k=1}^d \frac{f(x + \mu e_k)-f(x - \mu e_k)}{2 \mu} e_k, \\
    \hat{\nabla}^2 f(\mu, x) &\coloneqq \sum_{k=1}^d \frac{f(x + \mu e_k)-2f(x)+f(x - \mu e_k)}{\mu^2} e_k .
\end{aligned}
\label{eq:estimators}
\end{equation}

We emphasize that many algorithms, such as \cite{tang2020distributed} and \cite{akhavan2021distributed}, require $2d$ function queries only for the gradient estimation. In this work instead, by querying the function also at the current iterate, i.e. $f(x)$, which is in any case necessary to evaluate the progress, we estimate also the diagonal of the Hessian almost for free.

\begin{remark}
Obtaining these estimates involves $2d+1$ function queries, which can be computationally demanding for high-dimensional problems.
Alternatively, one could resort to randomized schemes in which directional derivatives are estimated along a small number of random vectors. However, using this type of estimators increases the total communication cost and does not guarantee that fewer total function evaluations are needed for convergence. As also argued in \cite{tang2020distributed}, which considers both a $2$-points randomized gradient estimator and a $2d$-points deterministic one, there exists a trade-off between the convergence rate and the ability to handle high-dimensional problems. In the interests of space, in this work we focus only on deterministic estimators, considering the use of stochastic estimators as a possible research avenue.
\end{remark}
Using Taylor expansions it is easily seen that in case of noiseless function evaluations the approximation errors of the estimators are bounded by
\begin{equation}
     \left\|  \hat{\nabla} f(\mu, x) - \nabla f(x) \right\| \leq \frac{\sqrt{d} L_2 \mu^2}{6},
\label{eq:gradient_est_error} 
\end{equation}
\begin{equation}
     \left\lvert \left[ \hat{\nabla}^2 f(\mu, x) \right]_{i} - \left[ \nabla^2 f(x)\right]_{ii} \right\rvert \leq \frac{L_3 \mu^2}{12} \quad i=1, \dots, d,
\label{eq:Hessian_est_error}
\end{equation}
and that the Jacobian of the gradient estimator takes the form
\begin{equation*}
\frac{\partial \hat{\nabla} f(x)}{\partial x} = \nabla^2 f(x) + R(\mu, x),
\qquad
\left\lvert \left[ R(\mu, x) \right]_{ij} \right\rvert
\leq \frac{L_3 \mu^2}{6}
\qquad \forall i,j,x.
\end{equation*}

\begin{remark}
The approximation errors of the estimators and the absolute value of the elements of the matrix $R(\mu, x)$ are bounded by functions of $\mu^2$. This is due to the symmetry of central difference-scheme, and shows that the estimates are very reliable.
\end{remark}

\noindent Since for a quadratic function $L_2 = L_3 = 0$, we have the following result:
\begin{lemma}
    In case of quadratic functions the gradient estimator and the Hessian estimator are exact, i.e. they coincide with the true gradient and with the diagonal of the true Hessian matrix, respectively.
    \label{lemma:estimators_exact_quadratic_case}
\end{lemma}

As shown below, both the estimators are Lipschitz continuous, and we will use this fact to prove the convergence of the algorithm.
\begin{equation*}
\begin{aligned}
\left\| \hat{\nabla} f(\mu, x) - \hat{\nabla} f(\mu, y) \right\| & \leq
\left( L_1 + \frac{\mu \sqrt{d} L_2}{2} \right) \left\| x-y \right\|, \\
\left\| \hat{\nabla}^2 f(\mu, x) - \hat{\nabla}^2 f(\mu, y) \right\| & \leq
\left( L_2 \sqrt{d} + \frac{\mu \sqrt{d} L_3}{3} \right) \left\| x-y \right\|.
\end{aligned}
\end{equation*}

\noindent Let $x^\star$ be the unique global minimizer of $f(x)$. The following Lemma states that for any value of $\mu$ there exists a unique point $\Gamma(\mu)$ for which the gradient estimator is the zero vector. Moreover, the distance between $\Gamma(\mu)$ and the optimal solution depends on $\mu^2$ and can be made arbitrarily small.

\begin{lemma}
Let $f(x)$ be a function which satisfies Assumption \ref{assumpt:convexity_smoothness} and $\hat{\nabla} f(\mu, x)$ be defined by \eqref{eq:estimators}. Then there exist $\bar{\mu}, \delta > 0$ and a continuously differentiable function $\Gamma:\ \mathbb{R} \to \mathbb{R}^d$ such that for all $\mu \leq \bar{\mu}$ there is a unique $x = \Gamma(\mu) \in \mathbb{R}^d$ such that $\left\| \Gamma(\mu) - x^\star \right\| \leq \delta \mu^2$ and $\hat{\nabla} f(\mu, \Gamma(\mu)) = 0$.
\label{lemma:implicit_function}
\end{lemma}

For reasons of space the proof of the Lemma, as well as the derivation of the bounds presented so far, can be found in Appendix \ref{appendix:estimators}. To lighten the notation, from now on we will denote the gradient and Hessian estimators as $\hat{\nabla} f(x)$ and $\hat{\nabla}^2 f(x)$, keeping the dependence from the parameter $\mu$ implicit.

\section{Proposed algorithm}

In this section we present ZO-JADE to efficiently solve the distributed convex optimization problem \eqref{eq:problem_formulation}. Our method can be seen as the ZO counterpart of a special case of \cite{varagnolo2015newton}, with the additional difference that consensus is performed also directly on the variable to be optimized. In particular, we employ a Newton-type update which only requires the diagonal of the Hessian, sometimes referred to as Jacobi descent \cite{bertsekas2015parallel}, from which the name ZO-JADE (JAcobi DEcentralized). This allows to benefit from the accelerated convergence rate which is typical of the second-order methods, while just estimating the curvature along the canonical basis. Also, since the Hessian estimate is a vector, the storage requirement and the communication overhead are linear in $d$, and computing the Hessian inverse is straightforward.

\renewcommand{\thealgorithm}{}
\begin{algorithm}
\caption{ZO-JADE}\label{alg}
\begin{algorithmic}
\Initialize{
     arbitrary $x_i(0) \in \mathbb{R}^d$ \\
     $y_i(0) = g_i(0) = \mathbb{0} \quad \forall i \in \mathcal{N}$ \\
    $z_i(0) = h_i(0) = \mathbb{0} \quad \forall i \in \mathcal{N}$ \\
    $\epsilon \in (0,1)$, consensus matrix $P$}
\For{$t = 1, 2, \dots$}
    \ForAll{$i \in \mathcal{N}$}
    
        \State Compute $\hat{\nabla} f_i(x_i(t-1))$, $\hat{\nabla}^2 f_i(x_i(t-1))$ using \eqref{eq:estimators}
        \State $g_i(t) = \hat{\nabla}^2 f_i(x_i(t-1)) \odot x_i(t-1) - \hat{\nabla} f_i(x_i(t-1))$
        \State $h_i(t) = \hat{\nabla}^2 f_i(x(t-1))$
        \State $y_i(t) = \sum_{j=1}^n p_{ij}[ y_j(t-1) + g_j(t) - g_j(t-1) ]$
        \State $z_i(t) = \sum_{j=1}^n p_{ij}[ z_j(t-1) + h_j(t) - h_j(t-1) ]$
        \State $x_i(t) = (1-\epsilon) \sum_{j=1}^n p_{ij} x_j(t-1) + \epsilon y_i(t) \oslash z_i(t)$
    \EndFor
\EndFor
\end{algorithmic}
\end{algorithm}

We now provide an intuitive explanation of the algorithm. The idea is to approximate each local $f_i(x)$ with a $d$-dimensional parabola whose axes are aligned with the canonical basis, e.g.
\[
f_i(x) \approx \hat{f_i}(x) \coloneqq \frac{1}{2}x^T \left( a_i \odot x \right) + b_i^T x + c_i,
\]
\[
n \hat{f}(x) = \sum_{i=1}^n \hat{f_i}(x) = \frac{1}{2}x^T \left( \sum_{i=1}^n a_i \odot x \right) + \left( \sum_{i=1}^n b_i^T \right) x + \sum_{i=1}^n c_i,
\]
for suitable $a_i \in \mathbb{R}^{d}$ with positive entries, $b_i \in \mathbb{R}^d$, $c_i \in \mathbb{R}$. At the current iterate $\{x_i\}_{i=1}^n$, we match the first and second derivative of the model with the estimates provided by the ZO oracles:
\[
\nabla \hat{f_i}(x_i) = a_i \odot x_i + b_i = \hat{\nabla} f_i(x_i),
\quad
\nabla^2 \hat{f_i}(x_i) = a_i = \hat{\nabla}^2 f_i(x_i).
\]
We then want to move towards the minimum of the approximate global function, namely
\begin{align}
\argmin_x \hat{f}(x) &= - \left( \frac{1}{n} \sum_{i=1}^n b_i \right) \oslash \left( \frac{1}{n} \sum_{i=1}^n a_i \right) \label{eq:min_parabola_model} \\
&=
\left( \sum_{i=1}^n \hat{\nabla}^2 f_i(x_i) \odot x_i - \hat{\nabla} f_i(x_i) \right) \oslash \left( \sum_{i=1}^n \hat{\nabla}^2 f_i(x_i) \right). \nonumber
\end{align}
To do so, we define the local variables $\in \mathbb{R}^d$
\begin{equation}
\begin{aligned}
g_i(t) &= \hat{\nabla}^2 f_i(x_i(t-1)) \odot x_i(t-1) - \hat{\nabla} f_i(x_i(t-1)), \\
h_i(t) &= \hat{\nabla}^2 f_i(x_i(t-1)),
\end{aligned}
\label{eq:def_gh}
\end{equation}
and we estimate the means in \eqref{eq:min_parabola_model} using gradient and Hessian tracking. As the name suggests, this technique introduces two additional consensus variables which track the average of the derivatives over the network:
\begin{equation}
\begin{aligned}
y_i(t) &= \sum_{j=1}^n p_{ij}( y_j(t-1) + g_j(t) - g_j(t-1) ), \\
z_i(t) &= \sum_{j=1}^n p_{ij}( z_j(t-1) + h_j(t) - h_j(t-1) ).
\end{aligned}
\label{eq:def_yz}
\end{equation}
Indeed, using recursion it is easy to verify that by initializing $y_i(0) = g_i(0), \ z_i(0) = h_i(0)\ \forall i$ it holds $\forall t>0$
\begin{equation}
\sum_{i=1}^n y_i(t) = \sum_{i=1}^n g_i(t),
\qquad \sum_{i=1}^n z_i(t) = \sum_{i=1}^n h_i(t).
\label{eq:tracking_property}
\end{equation}
Assuming now that the mean of $\{x_i\}_{i=1}^n$ is almost constant, the local variables $\{y_i\}_{i=1}^n$ and $\{z_i\}_{i=1}^n$ will eventually reach consensus, becoming equal to the numerator and denominator of \eqref{eq:min_parabola_model}.
The almost stationarity of the mean of $\{x_i\}_{i=1}^n$ is satisfied by adopting the time-scale separation framework. According to the latter, a sub-system can be regarded at steady state, provided that it evolves sufficiently fast compared to the rest of the system. Equivalently, a fast system can consider a slow one as a fixed constraint. In our case, we need the consensus \eqref{eq:def_yz} to be part of the fast dynamics, while the dynamics
\begin{equation*}
\frac{1}{n} \sum_{i=1}^n x_i(t) = \frac{1}{n} \sum_{i=1}^n \left[ (1-\epsilon) \sum_{j=1}^n p_{ij} x_j(t-1) + \epsilon y_i(t) \oslash z_i(t) \right]
\label{eq:update_rule_x}
\end{equation*}
must be sufficiently slow. We will prove that there exists a positive scalar $\bar{\epsilon}$ such that for any $\epsilon \leq \bar{\epsilon}$ the time-scales separation is verified and the algorithm converges exponentially fast. Moreover, property \eqref{eq:tracking_property} guarantees that we are moving towards the minimum \eqref{eq:min_parabola_model} of the global objective, allowing the convex combination coefficient $\epsilon$ to be constant. \\
Numerical simulations show that the consensus on the current iterates in the update rule of $x_i(t)$, which is not present in \cite{varagnolo2015newton}, brings significant performance improvements to the algorithm, at the price of an additional communication variable. However, the main concern in the zeroth-order setting is to minimize the number of function queries, which in our case is directly proportional to the number of iterations. Accordingly, the cost of the communication overhead becomes negligible when compared to the gain in terms of convergence rate.

\section{Convergence analysis}
In this section we prove the semi-global exponential convergence of ZO-JADE to a solution arbitrarily close to the true minimum of the optimization problem \eqref{eq:problem_formulation}.
Compared to \cite{varagnolo2015newton}, both the additional consensus step on $x$ and the presence of inexact derivative estimators pose significant challenges and require careful treatment. In particular, the former makes the algorithm incompatible with the time-scales separation principle in \cite{bof2018lyapunov}, which was needed in the convergence proof. For this reason, in this work we extend that principle and provide the following more general result, whose proof is given in Appendix \ref{appendix:separation_time_scales}.

\begin{theorem} [Semi-global exponential stability]

Consider the system
\begin{equation*}
\left\{\begin{array}{l}
x(0)=x_{0},\ y(0)=y_{0} \\
x(k+1)=x(k)+\epsilon \phi(x(k), y(k)) \\
y(k+1)=\varphi(y(k), x(k)) + \epsilon \Psi(y(k), x(k))
\end{array}\right.
\end{equation*}

and assume that is well defined for any generic variables $x \in \mathbb{R}^{n}, y \in \mathbb{R}^{n}$ and $k \in \mathbb{N}$. Also assume that the following assumption are satisfied for all $x \in \mathbb{R}^{n}, y \in \mathbb{R}^{n}, k \in \mathbb{N}$ :

\begin{enumerate}
\item The maps $\varphi(y,x)$, $\phi(x,y)$ and $\Psi(y,x)$ are locally uniformly Lipschitz.

\item The map $\varphi(y, x)$ satisfies all conditions of Theorem $6.3$ in \cite{bof2018lyapunov} and in particular there is function $y^{*}(x)$ such that $y^{*}(x)=\varphi\left(y^{*}(x), x\right)$ $\forall x$.

\item At the origin, $\phi\left(0, y^{*}(0)\right)=0$ and $\Psi\left(y^{*}(0), 0\right)=0$.

\item There exists a twice differentiable function $V(x)$ such that
\begin{equation*}
\begin{gathered}
c_{1}\|x\|^{2} \leq V(x) \leq c_{2}\|x\|^{2}, 
\qquad
\frac{\partial V}{\partial x} \phi\left(x, y^{*}(x)\right) \leq-c_{3}\|x\|^{2}, 
\qquad
\left\|\frac{\partial V}{\partial x}\right\| \leq c_{4}\|x\|.
\end{gathered}
\end{equation*}

\end{enumerate}

Then for each $r>0$, there exist $\epsilon_{r}, C_{r}, \gamma_{r}$ possibly function of $r$, such that for all $\left\|y_{0}-y^{*}(0)\right\|^{2}+\left\|x_{0}\right\|^{2}<r^{2}$ and $\epsilon \in\left(0, \epsilon_{r}\right)$, we have
\begin{equation*}
\|x(k)\|^{2} \leq C_{r}\left(1-\epsilon \gamma_{r}\right)^{k}
\end{equation*}

\label{theorem:sep_time_scales}
\end{theorem}

Below we present the main result, which provides theoretical guarantees on the stability and performance of the proposed algorithm.

\begin{theorem} [Convergence of ZO-JADE]
Under Assumption \ref{assumpt:convexity_smoothness}, there exists a positive $\bar{\epsilon}$ such that for any $\epsilon \in [0,\ \bar{\epsilon}]$ ZO-JADE is guaranteed to converge semi-globally and exponentially fast to $\Gamma(\mu)$.
\label{theorem:convergence_ZO_JADE}
\end{theorem}

\begin{proof}
The proof is based on a reformulation of our algorithm which leads to a discrete-time system compatible with Theorem \ref{theorem:sep_time_scales}.
First, we rewrite the iterations of ZO-JADE to obtain an equivalent autonomous system with the same structure as the one considered by Theorem \ref{theorem:sep_time_scales}. In the second part of the proof we demonstrate that the system satisfies all the assumptions of Theorem \ref{theorem:sep_time_scales}, and invoking the latter we get the desired result.
To simplify the notation let us stack the local variables and define $ x(t) = [ x_1(t), \dots, x_n(t) ]^T \in \mathbb{R}^{n \times d} $. Similarly, define $g(t),h(t),y(t),z(t), \hat{\nabla} f(x(t)), \hat{\nabla}^2 f(x(t)) \in \mathbb{R}^{n \times d}$.
We introduce the auxiliary variables
\begin{equation}
    \bar{x}(t) = \frac{1}{n}  \mathbb{1} \mathbb{1}^T x(t),
    \qquad
    \tilde{x}(t) = x(t) - \bar{x}(t),
\label{eq:x_bar_tilde}
\end{equation}
which are the mean of $x(t)$ and the displacement from it. This allows us to rewrite the algorithm as a dynamical system and to decouple the fast dynamics
\begin{equation*}
\begin{aligned}
g(t) &= \varphi_g = \hat{\nabla}^2 f(x(t-1)) \odot x(t-1) - \hat{\nabla} f(x(t-1)), \\
h(t) &= \varphi_h = \hat{\nabla}^2 f(x(t-1)), \\
y(t) &= \varphi_y = P \left( y(t-1) + g(t) - g(t-1) \right), \\
z(t) &= \varphi_z = P \left( z(t-1) + h(t) - h(t-1) \right), \\
\tilde{x}(t) &= \left( I - \frac{1}{n} \mathbb{1} \mathbb{1}^T \right) \left[ (1- \epsilon) P x(t-1) + \epsilon y(t) \oslash z(t) \right]
\\
&= \left( P - \frac{1}{n} \mathbb{1} \mathbb{1}^T \right)
x(t-1) + \epsilon \left[ \left( I - \frac{1}{n} \mathbb{1} \mathbb{1}^T \right) y(t) \oslash z(t) - \left( P - \frac{1}{n} \mathbb{1} \mathbb{1}^T \right)
x(t-1) \right]
\\
&= \left( P - \frac{1}{n} \mathbb{1} \mathbb{1}^T \right)
\tilde{x}(t-1) + \epsilon \left[ \left( I - \frac{1}{n} \mathbb{1} \mathbb{1}^T \right) y(t) \oslash z(t) - \left( P - \frac{1}{n} \mathbb{1} \mathbb{1}^T \right)
\tilde{x}(t-1) \right]
\\
&=  \varphi_{\tilde{x}} + \epsilon \Psi_{\tilde{x}}
\end{aligned}
\label{eq:dyn_sys_fast}
\end{equation*}
from the slow dynamics
\begin{equation*}
\begin{aligned}
\bar{x}(t) &= \frac{1}{n} \mathbb{1} \mathbb{1}^T \left[ (1- \epsilon) P x(t-1) + \epsilon y(t) \oslash z(t) \right]
\\
&= \bar{x}(t-1) + \epsilon \left( \frac{1}{n}  \mathbb{1} \mathbb{1}^T [y(t) \oslash z(t)] - \bar{x}(t-1) \right) \\
&= \bar{x}(t-1) + \epsilon \phi_{\bar{x}},
\end{aligned}
\label{eq:dyn_sys_slow}
\end{equation*}
where we used the fact that 
\begin{equation*}
\begin{aligned}
\left( P - \frac{1}{n} \mathbb{1} \mathbb{1}^T \right) x(t-1) &= \left( P - \frac{1}{n} \mathbb{1} \mathbb{1}^T \right) \left[ \bar{x}(t-1) + \tilde{x}(t-1) \right]
\\
&= \underbrace{ \left( P - \frac{1}{n} \mathbb{1} \mathbb{1}^T \right) \frac{1}{n}  \mathbb{1} \mathbb{1}^T }_{ \mathbb{0} \mathbb{0}^T} x(t-1) + \left( P - \frac{1}{n} \mathbb{1} \mathbb{1}^T \right) \tilde{x}(t-1).
\end{aligned}    
\end{equation*}
Since $x(t) = \bar{x}(t) + \tilde{x}(t)$ and $g(t),h(t),y(t),z(t)$ ultimately depend only on quantities at time $t-1$, the system is autonomous and can be further rewritten as
\begin{equation*}
\left\{\begin{array}{l}
\bar{x}(k)=\bar{x}(k-1)+\epsilon \phi \left( \bar{x}(k-1), \xi(k-1) \right) \\
\xi(k)=\varphi \left(\xi(k-1), \bar{x}(k-1) \right) + \epsilon \Psi \left(\xi(k-1), \bar{x}(k-1) \right)
\end{array}\right.
\end{equation*}
where we defined
\begin{equation*}
\begin{aligned}
\xi &= \left\{ g, h, y, z, \tilde{x} \right\} \\
\varphi \left(\xi(k-1), \bar{x}(k-1) \right) &= \varphi_g + \varphi_h + \varphi_y + \varphi_z + \varphi_{\tilde{x}} \\
\Psi \left(\xi(k-1), \bar{x}(k-1) \right) &= \Psi_{\bar{x}} \\
\phi \left(\bar{x}(k-1), \xi(k-1) \right) &= \phi_{\bar{x}}
\end{aligned}    
\end{equation*}
We are in the position to apply Theorem \ref{theorem:sep_time_scales}, which guarantees the exponential stability of the system for a suitable choice of the parameter $\epsilon$.
We now show that all the assumptions of such Theorem are verified. The fact that the gradient and Hessian estimators are differentiable and Lipschitz continuous guarantees that also the functions $\varphi$, $\phi$ and $\Psi$ are continuously differentiable and globally uniformly Lipschitz. Considering the boundary-layer system $\xi(k)=\varphi \left(\xi(k-1), \bar{x} \right)$, obtained by setting $\epsilon = 0$, we note that $\tilde{x}$ vanishes to zero and all the other components of $\xi$ reach consensus. In particular, the steady-state values of $y,z$ are
\begin{equation*}
\begin{aligned}
    y(t) &\overset{\eqref{eq:tracking_property}}{=} \frac{1}{n} \mathbb{1} \mathbb{1}^T g(t) \overset{\eqref{eq:def_gh}}{=} \hat{\nabla}^2 f(\bar{x}(t-1)) \odot \bar{x}(t-1) - \hat{\nabla} f(\bar{x}(t-1)), \\
    z(t) &\overset{\eqref{eq:tracking_property}}{=} \frac{1}{n} \mathbb{1} \mathbb{1}^T h(t) \overset{\eqref{eq:def_gh}}{=} \hat{\nabla}^2 f(\bar{x}(t-1)).
\end{aligned}
\end{equation*}
This verifies the assumption that there exists $\xi^\star$ such that $\xi^\star(\bar{x})=\varphi \left(\xi^\star(\bar{x}), \bar{x} \right)$ for any choice of $\bar{x}$, and ensures that $\Psi\left(\xi^{*}(\mathbb{0}), \mathbb{0}\right)=\mathbb{0}$. Noticing that
\begin{equation*}
\phi \left(\bar{x}, \xi^\star(\bar{x}) \right) = - \hat{\nabla} f(\bar{x}) \oslash \hat{\nabla}^2 f(\bar{x})
\end{equation*}
and recalling Lemma \ref{lemma:implicit_function}, we have $\phi \left(\Gamma(\mu), \xi^\star(\Gamma(\mu)) \right) = \mathbb{0}$. To satisfy assumption \textit{(iii)} of Theorem \ref{theorem:sep_time_scales} we assume without loss of generality that the target solution $\Gamma(\mu)$, which is arbitrarily close to the global minimum of $f(x)$, coincides with the origin. This corresponds to solving an equivalent problem in which the objective function is translated by an offset.

Finally, it is sufficient to choose as Lyapunov function $V(\bar{x})= \left\| \hat{\nabla}f(\bar{x}) \right\|^2$, which is always strictly positive except for $V(\mathbb{0}) = 0$. Indeed, it is possible to deduce the following bounds for $V(\bar{x})$, whose proofs are contained in Appendix \ref{appendix:estimators}.
\begin{equation*}
\begin{aligned}
V(\bar{x}) & \leq \left( L_1 + \frac{\mu \sqrt{d} L_2}{2} \right)^2 \left\| \bar{x} -\Gamma(\mu) \right\|^2, \\
V(\bar{x}) & \geq \left( m^2 - 2 L_1 \frac{\mu^2 L_3}{6} -\left( \frac{\mu^2 L_3}{6} \right)^2 \right) \left\| \bar{x} - \Gamma(\mu) \right\|^2, \\
\left\| \frac{\partial V}{\partial \bar{x}} \right\|
& \leq \left(2 L_1 + \frac{\mu L_3 d}{3} \right) \left( L_1 + \frac{\mu \sqrt{d} L_2}{2} \right) \left\| \bar{x} -\Gamma(\mu) \right\|, \\
\frac{\partial V}{\partial \bar{x}} \phi \left(\bar{x}, \xi^\star(\bar{x}) \right)
& \leq \left( \frac{2 d \mu^2 L_3}{2m - L_3 \mu^2} -\frac{12m}{12 L_1 + L_3 \mu^2}  \right) \left( L_1 + \frac{\mu \sqrt{d} L_2}{2} \right)^2 \left\| \bar{x} -\Gamma(\mu) \right\|^2.
\end{aligned}
\end{equation*}
To ensure that the lower bound on $V(\bar{x})$ is not trivial and that we are moving along a descent direction, when $L_3 \neq 0$ we require that $\mu \leq \min \left\{ \mu_1, \mu_2 \right\}$, where
\begin{equation*}
\mu_1 = \sqrt{ \frac{6 \left( \sqrt{L_1^2 + m^2} - L_1 \right) }{d L_3} },
\qquad
\mu_2 = \sqrt{ 3 \frac{ \sqrt{4d^2 L_1^2 + m^2 + 4dL_1 m + 8 m^2 d} - 2d L_1 - m }{d L_3} }.
\end{equation*}
Since all the required conditions are verified, the proof is complete.

\end{proof}

\begin{remark}
We recall that while the algorithm converges to $\Gamma(\mu)$, in virtue of Lemma \ref{lemma:implicit_function} it is possible to make such point arbitrarily close to the actual global minimum of $f(x)$ by reducing $\mu$. Therefore, the quantization parameter $\mu$ in the derivative estimators determines the accuracy of the solution reached by our algorithm.
\end{remark}

\begin{remark}
Since the separation of time-scales is mainly an existential theorem, it is difficult to provide a bound on the parameter $\bar{\epsilon}$ or an explicit convergence rate. Even if this was possible, very conservative values of little practical use would be obtained. This is common to most works, where usually in the numerical simulations the hyperparameters are tuned by trial-and-error rather than according to possibly available theoretical bounds.
\end{remark}


\section{Numerical results}

\begin{figure}[t!]
  \centering
 
     \begin{subfigure}[b]{\columnwidth}
         \centering
           \parbox{0.85\columnwidth}{
          \includegraphics[width=0.90\columnwidth]{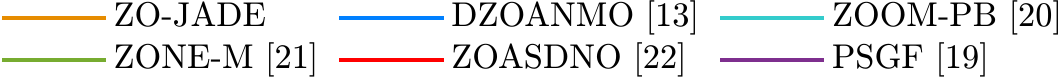}
          }
     \end{subfigure}
     
     \vspace{0.2cm}
     \hfill
     \begin{subfigure}[b]{\columnwidth}
         \centering
         \parbox{0.9\columnwidth}{
              \caption{Ridge regression on the dataset \cite{Coraddu2013Machine}.}
              \includegraphics[width=0.95\columnwidth]{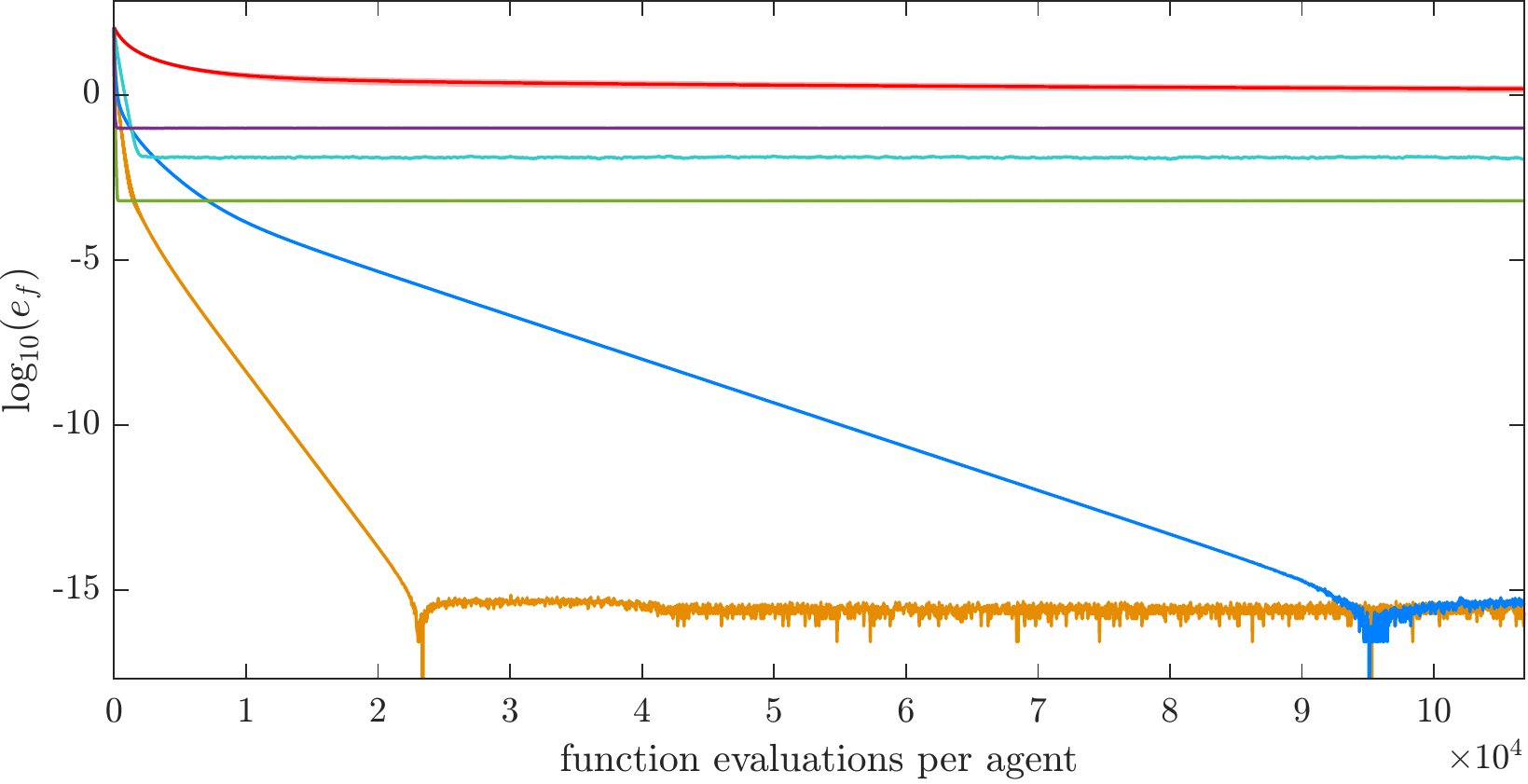}
         }  
         \label{fig:ridge_reg}
     \end{subfigure}

    \vspace{0.2cm}
     \hfill
     \begin{subfigure}[b]{\columnwidth}
          \centering
          \parbox{0.9\columnwidth}{
              \caption{Classification via logistic regression on the dataset MNIST.}
              \includegraphics[width=0.95\columnwidth]{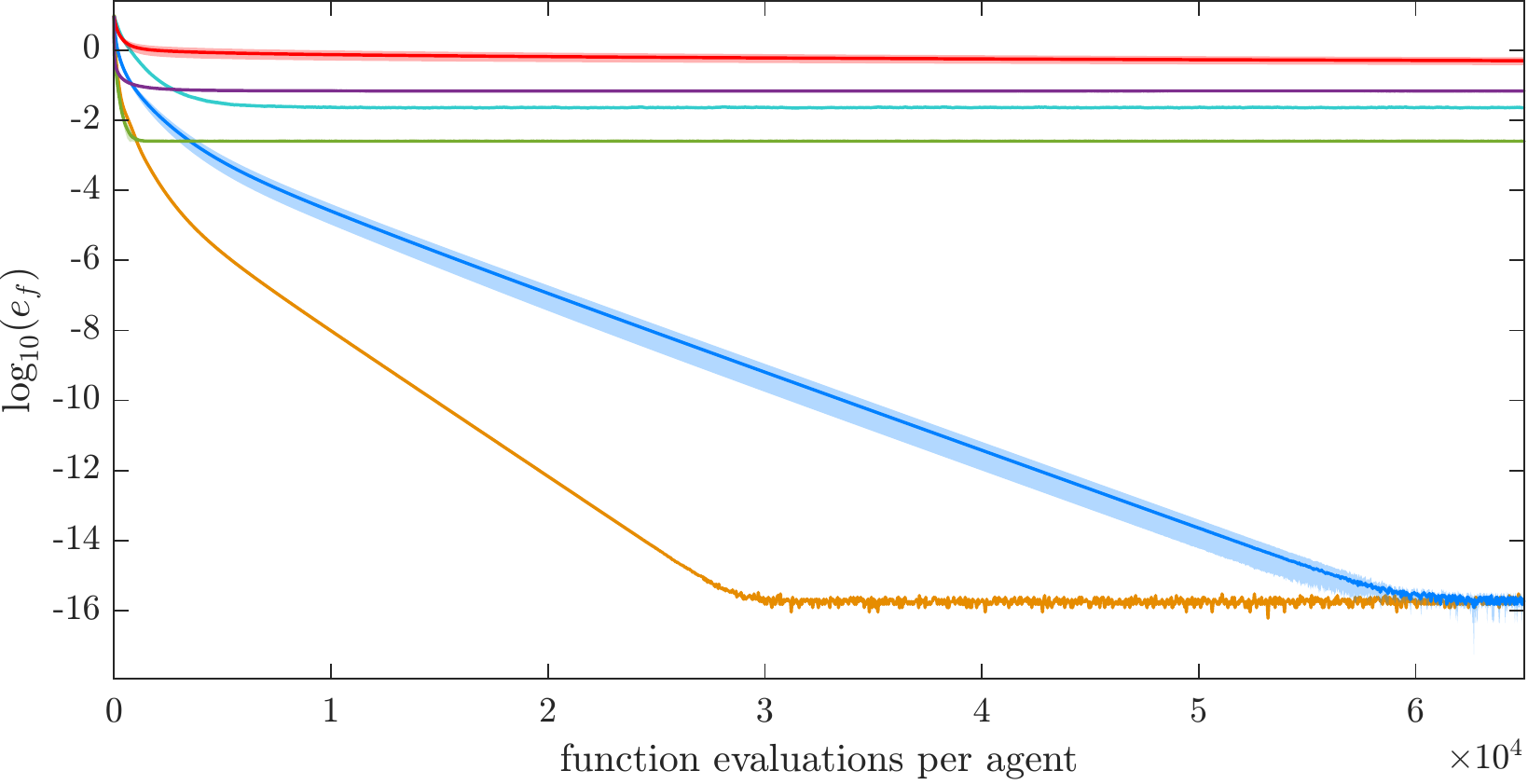}
          }  
          \label{fig:logreg_MNIST}
     \end{subfigure}
     \caption{Evolution of the loss function versus number of function queries. The plots show the average performance over multiple runs starting from different points, where the solid lines are the mean trajectories and the shaded areas represent the standard deviation.}
  
\end{figure}

Let us present some numerical results, which validate the effectiveness of ZO-JADE and show its superiority with respect to several state-of-the-art ZO algorithms for multi-agent optimization. In particular, we compare our method with the following distributed ZO algorithms: PSGF \cite{yuan2015gradient}, which employs a push-sum strategy; the accelerated version of ZOOM, ZOOM-PB \cite{zhang2022zeroth}, which requires the Laplacian matrix of the network and uses a powerball function; ZONE-M \cite{hajinezhad2017zeroth}, a primal-only version of the Method of Multipliers; the primal-dual method proposed in \cite{yi2022zeroth}, dubbed here as ZOASDNO using the capital letters in the title of the paper, based on the Laplacian framework; the algorithm in \cite{tang2020distributed} using a $2d$-points gradient estimator and gradient tracking, denoted here with the acronym DZOANMO. For all these algorithms, we tune the possible parameters as suggested in the respective papers. For ZOOM-PB and ZONE-M, whose randomized gradient estimators consider only a subset of coordinates, we test different numbers of search directions and show only the best performances obtained. We average over multiple runs, starting from different randomly generated initial points $x_i(0)$, which are the same for all the algorithms.

We consider a network of $n=20$ agents satisfying Assumption \ref{assumpt:graph}, where the consensus matrix $P$ is built using the Metropolis-Hastings weights, and we test the algorithms on two types of cost functions. In the first scenario we consider quadratic objectives, and specifically we perform ridge regression on the condition-based maintenance dataset \cite{Coraddu2013Machine} to predict the health state of a machine. In the second test we consider more general convex functions, addressing a binary classification problem via logistic regression. In particular we run one-vs-all classification, in which one target class needs to be distinguished from all the others, on the well-known MNIST dataset \cite{deng2012mnist}. Adopting the approach followed by \cite{fabbro2022newton} on a similar dataset, we apply PCA to reduce the size of the problem to $d=20$ and provide each agent with a balanced set $\mathcal{D}_i= \{ s_1,\dots, s_{|\mathcal{D}_i|} \}$ where half of the samples belong the target class, and the other half to the remaining classes, equally mixed. The objective functions are the regularized log-loss cost functions
\[
f_i(x, \mathcal{D}_i) = \frac{1}{\mathcal{D}_i} \sum_{k=1}^{|\mathcal{D}_i|} \log\left( 1 + \exp\left( - l_k [s_k^T \ 1] x \right) \right) + \frac{w}{2} \left\| x \right\|^2,
\]
where $l_k \in \{-1,1\}$ is the label corresponding to the predictor sample $s_k \in \mathbb{R}^{d-1}$ and the quadratic term with $w > 0$ ensures strong convexity.
\\
To take into account the possible differences between the local solutions $x_i$, we evaluate the performance of the algorithms using the loss function
\begin{equation}
e_f = \frac{ \frac{1}{n} \sum_{i=1}^n f(x_i) - f(x^\star) } {|f(x^\star)|},
\end{equation}
where $x^\star$ is the unique minimizer of $f$. The speed of convergence is measured in terms of number of function evaluations, which in the zeroth-order optimization field are assumed to be the most expensive computations.

The numerical simulations clearly show that ZO-JADE outperforms the other algorithms in the considered scenarios, as it achieved an equivalent or lower loss value using a smaller amount of function evaluations. This suggests that our method is suitable both in the case where a very precise solution is desired and in the case where a limited amount of function queries are available.

\section{Conclusion}
We presented a novel zeroth-order algorithm for distributed convex optimization, which is the first of its kind to exploit the second order derivative. We provided an intuitive explanation of our method and we formally proved its semi-global stability and exponentially fast convergence. The theoretical results are validated by numerical simulations on publicly available real-world datasets, in which ZO-JADE is shown to outperform several zeroth-order distributed algorithms present in the literature. We foresee interesting possibilities for future work, such as extensions to online and stochastic cases and the use of randomized derivative estimators.


\appendix

\section{Separation of the time-scales}
\label{appendix:separation_time_scales}
In this appendix we use Lyapunov theory to analyze the stability of a particular class of discrete-time systems. In particular we provide a proof of Theorem \ref{theorem:sep_time_scales}, which extends Proposition 8.1 in \cite{bof2018lyapunov} to include the case in which also the fast dynamics depends on the step-size parameter $\epsilon$, allowing the application of the time-scales separation principle to a broader class of problems. We use the notation and terminology introduced by \cite{bof2018lyapunov} and we refer to results contained therein. For an equivalent treatment of the continuous-time case we refer the reader to the chapter `Singular Perturbations' in \cite{khalil2015nonlinear}.

\begin{proof}
Assume that all the assumptions in the statement of Theorem \ref{theorem:sep_time_scales} are verified. Let $y^{\prime}(k)$ be defined as

\begin{equation*}
y^{\prime}(k) \coloneqq y(k)-y^{*}(x(k))
\end{equation*}

We can write that
\begin{equation*}
\begin{aligned}
y^{\prime}(k+1) & =y(k+1)-y^{*}(x(k+1)) \\
& =\varphi(y(k), x(k)) + \epsilon \Psi(y(k), x(k)) -y^{*}(x(k+1)) \\
& =\varphi\left(y^{\prime}(k)+y^{*}(x(k)), x(k)\right) + \epsilon \Psi(y^{\prime}(k)+y^{*}(x(k), x(k)) \\
& \quad -y^{*}\left(x(k)+\epsilon \phi\left(x(k), y^{\prime}(k)+y^{*}(x(k))\right)\right)
\end{aligned}
\end{equation*}

Then the dynamics of the original system can be written in this new coordinate system as:
\begin{equation*}
\left\{
\begin{aligned}
x(k+1) &=x(k)+\epsilon \phi\left(x(k), y^{\prime}(k)+y^{*}(x(k))\right) \\
y^{\prime}(k+1) &=\varphi\left(y^{\prime}(k)+y^{*}(x(k)), x(k)\right) + \epsilon \Psi(y(k), x(k)) -y^{*}(x(k+1)) \\
& \quad -y^{*}\left(x(k)+\epsilon \phi\left(x(k), y^{\prime}(k)+y^{*}(x(k))\right)\right)
\end{aligned}\right.
\end{equation*}

By assumption \textit{(ii)}, then there exists a Lyapunov function $W(y, x)$ such that
\begin{equation*}
\begin{gathered}
b_{1}\left\|y^{\prime}\right\|^{2} \leq W\left(y^{\prime}, x\right) \leq b_{2}\left\|y^{\prime}\right\|^{2} \\
W\left(\varphi\left(y^{\prime}+y^{*}(x), x\right)-y^{*}(x), x\right)-W\left(y^{\prime}, x\right) \leq-b_{3}\left\|y^{\prime}\right\|^{2} \\
\left|W\left(y_{1}^{\prime}, x\right)-W\left(y_{2}^{\prime}, x\right)\right| \leq b_{4}\left\|y_{1}^{\prime}-y_{2}^{\prime}\right\|\left(\left\|y_{1}^{\prime}\right\|+\left\|y_{2}^{\prime}\right\|\right) \\
\left|W\left(y^{\prime}, x_{1}\right)-W\left(y^{\prime}, x_{2}\right)\right| \leq b_{5}\left\|y^{\prime}\right\|^{2}\left\| x_{1}-x_{2} \right\|
\end{gathered}
\end{equation*}

By assumption \textit{(iv)}, Proposition $7.3$ in \cite{bof2018lyapunov} guarantees the existence of a Lyapunov function $V^{\prime}(x)$ and constant $\epsilon_{c} \in(0,1]$ such that
\begin{equation*}
\begin{gathered}
a_{1}\|x\|^{2} \leq V^{\prime}(x) \leq a_{2}\|x\|^{2} \\
V^{\prime}\left(x+\epsilon \phi \left( x, y^{*}(x)\right)  \right) -V^{\prime}(x) \leq-\epsilon a_{3}\|x\|^{2}, \quad \epsilon \in\left(0, \epsilon_{c}\right) \\
\left|V^{\prime}(x)-V^{\prime}\left(x^{\prime}\right)\right| \leq a_{4}\left\|x-x^{\prime}\right\|\left(\|x\|+\left\|x^{\prime}\right\|\right)
\end{gathered}
\end{equation*}

Let us define the extended state vector $z=\left[x^{T} \ y^{\prime T}\right]^{T}$ and consider now the global Lyapunov function:
\begin{equation*}
U\left(x, y^{\prime}\right)=V^{\prime}(x)+W\left(y^{\prime}, x\right)
\end{equation*}

which has the property
\begin{equation*}
\min \left\{a_{1}, b_{1}\right\}\left(\|x\|^{2}+\left\|y^{\prime}\right\|^{2}\right) \leq U\left(x, y^{\prime}\right) \leq \max \left\{a_{2}, b_{2}\right\}\left(\|x\|^{2}+\left\|y^{\prime}\right\|^{2}\right)
\end{equation*}

Now, defining $\alpha \coloneqq \min \left\{a_{1}, b_{1}\right\}$ and $\beta \coloneqq \max \left\{a_{2}, b_{2}\right\}$ let us consider the sets
\begin{equation*}
\begin{aligned}
\mathcal{B}_{r} & \coloneqq \left\{(x, y) \mid\|x\|^{2}+\left\|y^{\prime}\right\|^{2} \leq r\right\} \\
\Omega_{r}(k) & \coloneqq \{(x, y) \mid U(x, y) \leq r \beta\}
\end{aligned}
\end{equation*}

\begin{equation*}
\begin{aligned}
\mathcal{B}_{r_{0}} \coloneqq \left\{(x, y) \mid\|x\|^{2}+\left\|y^{\prime}\right\|^{2} \leq r \frac{\beta}{\alpha}\right\}=\left\{(x, y)\|\| x\left\|^{2}+\right\| y^{\prime} \|^{2} \leq r_{0}\right\}, r_{0} \coloneqq r \frac{\beta}{\alpha}
\end{aligned}
\end{equation*}

which are closed and compact by continuity. Also note that
\begin{equation*}
\left(x_{0}, y_{0}-y^{*}(0)\right)=\left(x_{0}, y_{0}^{\prime}\right) \in \mathcal{B}_{r} \subset \Omega_{r}(k) \subset \mathcal{B}_{r_{0}}, \forall k \geq 0
\end{equation*}

where we used the fact that $y_{0}^{\prime}=y_{0}-y^{*}(0)$. The following basic bounds follow immediately from the Lipschitz and vanishing properties of the various functions when the state is restricted to belong to the compact set $\left(x, y^{\prime}\right) \in \mathcal{B}_{r_{0}}$.

\begin{equation*}
\begin{aligned}
\left\|\phi\left(x, y^{*}(x)\right)\right\| & \leq \ell_{1}\|x\| \\
\left\|\phi\left(x, y^{\prime}+y^{*}(x)\right)-\phi\left(x, y^{*}(x)\right)\right\| & \leq \ell_{2}\left\|y^{\prime}\right\| \\
\left\|\varphi\left(y^{\prime}+y^{*}(x), x\right)-y^{*}(x)\right\| & =\left\|\varphi\left(y^{\prime}+y^{*}(x), x\right)-\varphi\left(y^{*}(x), x\right)\right\| \\
& \leq \ell_{3}\left\|y^{\prime}\right\| \\
\left\|y^{*}\left(x+\epsilon \phi\left(x, y^{\prime}+y^{*}(x)\right)\right)-y^{*}(x)\right\| & \leq \epsilon \ell_{4}\left\|\phi\left(x, y^{\prime}+y^{*}(x)\right)\right\| \\
& \leq \epsilon \ell_{4}\left(\ell_{1}\|x\|+\ell_{2}\left\|y^{\prime}\right\|\right) \\
\left\|\Psi\left(x, y^{\prime}+y^{*}(x)\right)\right\| &\leq \left\|\Psi\left(x, y^{\prime}+y^{*}(x)\right) - \Psi\left(x, y^{*}(x)\right) \right\| + \left\| \Psi\left(x, y^{*}(x)\right) \right\| \\
& \leq \ell_{5}\|y^{\prime}\| + \ell_{6} \|x\| 
\end{aligned}
\end{equation*}

which also imply
\begin{equation*}
\begin{aligned}
\left\|\phi\left(x, y^{\prime}+y^{*}(x)\right)\right\| & \leq \ell_{2}\left\|y^{\prime}\right\|+\ell_{1}\|x\| \\
\left\|x+\epsilon \phi\left(x, y^{\prime}+y^{*}(x)\right)\right\| & \leq\|x\|+\epsilon\left(\ell_{1}\|x\|+\ell_{2}\left\|y^{\prime}\right\|\right) \\
\left\|\varphi\left(y^{\prime}+y^{*}(x), x\right) + \epsilon \Psi\left(x, y^{\prime}+y^{*}(x) \right) -y^{*}(x)\right\|
& \leq \left\|\varphi\left(y^{\prime}+y^{*}(x), x\right)-y^{*}(x)\right\| + \epsilon \left\|\Psi\left(x, y^{\prime}+y^{*}(x)\right)\right\| \\
& \leq (\ell_{3} + \epsilon \ell_{5})\left\|y^{\prime}\right\| + \epsilon \ell_{6} \|x\|
\end{aligned}
\end{equation*}

To simplify the notation let us indicate $x=x(k)$, $x^{+}=x(k+1)$ and $ \chi=x+\epsilon \phi\left(x, y^{*}(x)\right)$. We first want to find an upper bound for the Lyapunov function relative to the slow dynamics:
\begin{equation*}
\begin{aligned}
\Delta V\left(y^{\prime}, x\right) &  \coloneqq V^{\prime}(x(k+1))-V^{\prime}(x(k)) \\
& =V^{\prime}\left(x^{+}\right)-V^{\prime}( \chi)+V^{\prime}(\chi)-V^{\prime}(x) \\
& \leq a_{4}\left\|x^{+}-\chi\right\|\left(\left\|x^{+}\right\|+\|\chi\|\right)-\epsilon a_{3}\|x\|^{2} \\
& \leq \epsilon a_{4} \ell_{2}\left\|y^{\prime}\right\|\left(2\left(1+\epsilon \ell_{1}\right)\|x\|+\epsilon \ell_{2}\left\|y^{\prime}\right\|\right)-\epsilon a_{3}\|x\|^{2} \\
& \leq \epsilon^{2} a_{4} \ell_{2}^{2}\left\|y^{\prime}\right\|^{2}+\epsilon 2 a_{4} \ell_{2}\left(1+\epsilon \ell_{1}\right)\|x\|\left\|y^{\prime}\right\|-\epsilon a_{3}\|x\|^{2} \\
& \leq A_{V}\|x\|^{2}+B_{V}\|x\|\|y\|+C_{V}\left\|y^{\prime}\right\|^{2}
\end{aligned}
\end{equation*}

with $A_{V}=-\epsilon a_{3}, B_{V}=\epsilon 2 a_{4} \ell_{2}\left(1+\epsilon \ell_{1}\right)$ and $C_{V}=\epsilon^{2} a_{4} \ell_{2}^{2}$, which, for suitable positive constants $v_{A_{1}}, v_{B_{1}}, v_{B_{2}}, v_{C_{1}}$, can be rewritten as $A_{V}=-\epsilon v_{A_{1}}, B_{V}=$ $\epsilon v_{B_{1}}+\epsilon^{2} v_{B_{2}}$ and $C_{V}=\epsilon^{2} v_{C_{1}}$.

To simplify the notation let us indicate $y^{\prime}=y^{\prime}(k), y^{\prime+}=y(k+1)$ and

\begin{equation*}
\begin{aligned}
\Delta W\left(y^{\prime}, x\right) & \coloneqq W\left( y^{\prime}(k+1), x(k+1)\right)-W\left(y^{\prime}(k), x(k)\right) \\
& \begin{rcases*}
=W\left( \varphi\left(y^{\prime}+y^{*}(x), x\right) + \epsilon \Psi\left(x, y^{\prime}+y^{*}(x) \right) -y^{*}\left(x^{+}\right), x^{+}\right) \\
\quad - W\left( \varphi\left(y^{\prime}+y^{*}(x), x\right) + \epsilon \Psi\left(x, y^{\prime}+y^{*}(x) \right) -y^{*}(x), x^{+}\right)
\end{rcases*} \Delta W_{1} \\
& \begin{rcases*}
\quad + W\left( \varphi\left(y^{\prime}+y^{*}(x), x\right) + \epsilon \Psi\left(x, y^{\prime}+y^{*}(x) \right) -y^{*}(x), x^{+}\right) \\
\quad -W\left( \varphi\left(y^{\prime}+y^{*}(x), x\right) + \epsilon \Psi\left(x, y^{\prime}+y^{*}(x) \right) -y^{*}(x), x\right)
\end{rcases*} \Delta W_{2} \\
& \begin{rcases*}
\quad + W\left( \varphi\left(y^{\prime}+y^{*}(x), x\right) + \epsilon \Psi\left(x, y^{\prime}+y^{*}(x) \right) -y^{*}(x), x\right) \\
\quad - W\left( \varphi\left(y^{\prime}+y^{*}(x), x\right) + \epsilon \Psi\left(x, y^{*}(x) \right) -y^{*}(x), x\right)
\end{rcases*} \Delta W_{3} \\
& \begin{rcases*}
\quad + W\left( \varphi\left(y^{\prime}+y^{*}(x), x\right) + \epsilon \Psi\left(x, y^{*}(x) \right) -y^{*}(x), x\right) \\
\quad - W\left(y^{\prime}, x\right)
\end{rcases*} \Delta W_{4}
\end{aligned}
\end{equation*}

Using the properties of the Lyapunov function $W$ and the several bounds obtained before we get:

\begin{equation*}
\begin{aligned}
\Delta W_{1} & \leq b_{4}\left\|y^{*}\left(x^{+}\right)-y^{*}(x)\right\|\left(\left\|\varphi\left(y^{\prime}+y^{*}(x), x\right) + \epsilon \Psi\left(x, y^{\prime}+y^{*}(x) \right) -y^{*}\left(x^{+}\right)\right\| \right. \\
& \quad \left. +\left\|\varphi\left(y^{\prime}+y^{*}(x), x\right) + \epsilon \Psi\left(x, y^{\prime}+y^{*}(x) \right) -y^{*}(x)\right\| \right) \\
& \leq b_{4} \epsilon \ell_{4}\left(\ell_{1}\|x\|+\ell_{2}\left\|y^{\prime}\right\|\right)\left(2\left\|\varphi\left(y^{\prime}+y^{*}(x), x\right) + \epsilon \Psi\left(x, y^{\prime}+y^{*}(x) \right) -y^{*}(x)\right\|+\left\|y^{*}\left(x^{+}\right)-y^{*}(x)\right\|\right) \\
& \leq b_{4} \epsilon \ell_{4}\left(\ell_{1}\|x\|+\ell_{2}\left\|y^{\prime}\right\|\right)\left(2 (\ell_{3} + \epsilon \ell_{5}) \left\|y^{\prime}\right\| + 2 \epsilon \ell_{6} \left\|x\right\| +\epsilon \ell_{4}\left(\ell_{1}\|x\|+\ell_{2}\left\|y^{\prime}\right\|\right)\right) \\
& \leq \epsilon b_{4} \ell_{4}\left[ \epsilon \ell_{1} \left(\ell_{1} \ell_{4} + 2\ell_{6} \right) \|x\|^{2}+\left(2 \ell_{1} \ell_{3} + 2 \epsilon \ell_{1} \ell_{5} +2 \epsilon \ell_{1} \ell_{2} \ell_{4} + 2 \epsilon \ell_{2} \ell_{6} \right)\left\|y^{\prime}\right\|\|x\|+\ell_{2}\left(2 \ell_{3} + 2 \epsilon \ell_{5}+\epsilon \ell_{2} \ell_{4}\right)\left\|y^{\prime}\right\|^{2}\right] \\
& \leq A_{W_{1}}\|x\|^{2}+B_{W_{1}}\|x\|\left\|y^{\prime}\right\|+C_{W_{1}}\left\|y^{\prime}\right\|^{2}
\end{aligned}
\end{equation*}

with $A_{W_{1}} \coloneqq \epsilon^{2} b_{4}  \ell_{1} \ell_{4} \left( \ell_{1} \ell_{4} + 2 \ell_{6} \right), B_{W_{1}} \coloneqq \epsilon b_{4} \ell_{4}\left(2 \ell_{1} \ell_{3} + 2 \epsilon \ell_{1} \ell_{5} +2 \epsilon \ell_{1} \ell_{2} \ell_{4}  + 2 \epsilon \ell_{2} \ell_{6} \right)$ and $C_{W_{1}} \coloneqq \epsilon b_{4} \ell_{2} \ell_{4} \left(2 \ell_{3} + 2 \epsilon \ell_{5} + \epsilon \ell_{2} \ell_{4}\right)$

\begin{equation*}
\begin{aligned}
\Delta W_{2} & \leq b_{5}\left\|\varphi\left(y^{\prime}+y^{*}(x), x\right) + \epsilon \Psi\left(x, y^{\prime}+y^{*}(x) \right) -y^{*}(x)\right\|^{2}\left\|x^{+}-x\right\| \\
& \leq b_{5} \left[ (\ell_{3} + \epsilon \ell_{5}) \left\|y^{\prime}\right\| + \epsilon \ell_{6} \left\|x\right\| \right]^2 \epsilon\left(\ell_{1}\|x\|+\ell_{2}\left\|y^{\prime}\right\|\right)
\end{aligned}
\end{equation*}

Now, considering that $\|x\|$ and $\|y\|$ are smaller than a $\bar{r}$ and that $y^{*}$ is a twice differentiable function of $x$, we have that $\left\|y^{*}\right\|$ is bounded itself by a $\tilde{r}$, and so

\begin{equation*}
\left\|y^{\prime}\right\| = \left\|y-y^{\star}(x)\right\| \leq \left\|y\right\| + \left\|y^{\star}(x)\right\| \leq \bar{r}+\tilde{r}
\end{equation*}

\begin{equation*}
\begin{aligned}
\Delta W_{2} & \leq b_{5} \left[ (\ell_{3} + \epsilon \ell_{5})^{2} \left\|y^{\prime}\right\|^2 + \epsilon^2 \ell_{6}^2 \left\|x\right\|^2 + 2\epsilon \ell_{6} (\ell_{3} + \epsilon \ell_{5}) \left\|y^{\prime}\right\| \left\|x\right\| \right] \epsilon\left[ \ell_{1} \bar{r} +\ell_{2} (\bar{r}+\tilde{r}) \right]
\\
&= A_{W_{2}}\|x\|^2 + B_{W_{2}}\|x\|\left\|y^{\prime}\right\| + C_{W_{2}}\left\|y^{\prime}\right\|^{2},
\end{aligned}
\end{equation*}

with $A_{W_{2}} \coloneqq \epsilon^3 b_5 \ell_{6}^2 \left[ \ell_{1} \bar{r} +\ell_{2} (\bar{r}+\tilde{r}) \right]$, 
$B_{W_{2}} \coloneqq \epsilon^2 b_{5} 2 \ell_{6} (\ell_{3} + \epsilon \ell_{5}) \left[ \ell_{1} \bar{r} +\ell_{2} (\bar{r}+\tilde{r}) \right]$ and $C_{W_{2}} \coloneqq \epsilon b_{5} (\ell_{3} + \epsilon \ell_{5})^{2} \left[ \ell_{1} \bar{r} +\ell_{2} (\bar{r}+\tilde{r}) \right]$. Using also Assumption \textit{(iv)}:
\begin{equation*}
\begin{aligned}
\Delta W_{3} & \leq b_{4} \epsilon \left\| \Psi\left(x, y^{\prime}+y^{*}(x) \right) - \Psi\left(x, y^{*}(x) \right) \right\|\left(\left\|\varphi\left(y^{\prime}+y^{*}(x), x\right) + \epsilon \Psi\left(x, y^{\prime}+y^{*}(x) \right) -y^{*}\left(x\right)\right\| \right. \\
& \quad \left. +\left\|\varphi\left(y^{\prime}+y^{*}(x), x\right) + \epsilon \Psi\left(x, y^{*}(x) \right) -y^{*}(x)\right\| \right) \\
&= b_{4} \epsilon \left\| \Psi\left(x, y^{\prime}+y^{*}(x) \right) \right\|\left(\left\|\varphi\left(y^{\prime}+y^{*}(x), x\right) + \epsilon \Psi\left(x, y^{\prime}+y^{*}(x) \right) -y^{*}\left(x\right)\right\| +\left\|\varphi\left(y^{\prime}+y^{*}(x), x\right) -y^{*}(x)\right\| \right) \\
& \leq \epsilon b_{4} \left( \ell_{5} \left\|y^{\prime}\right\| + \ell_{6} \left\|x\right\| \right) \left[ (\ell_{3} + \epsilon \ell_{5}) \left\|y^{\prime}\right\| + \epsilon \ell_{6} \left\|x\right\| + \ell_{3} \left\|y^{\prime}\right\| \right] \\
&= A_{W_{3}}\|x\|^2 + B_{W_{3}}\|x\|\left\|y^{\prime}\right\| + C_{W_{3}}\left\|y^{\prime}\right\|^{2}
\end{aligned}
\end{equation*}

with $A_{W_{3}} \coloneqq \epsilon^2 b_4 \ell_{6}^2$, 
$B_{W_{3}} \coloneqq \epsilon 2 b_4 \ell_{6} (\ell_{3} + \epsilon \ell_{5}) $ and $C_{W_{3}} \coloneqq \epsilon b_4 \ell_{5} (2\ell_{3} + \epsilon \ell_{5})$.

\begin{equation*}
\begin{aligned}
\Delta W_{4} &= W\left(\varphi\left(y^{\prime}+y^{*}(x), x\right)-y^{*}(x), x\right)-W\left(y^{\prime}, x\right) \\
& \leq-b_{3}\left\|y^{\prime}\right\|  = -C_{W_{4}}\left\|y^{\prime}\right\|^{2},
\end{aligned}
\end{equation*}

with $C_{W_{4}} \coloneqq b_{3}$. Summing up all the previous inequalities, we have

\begin{equation*}
\Delta W\left(y^{\prime}, x\right) \leq A_{W}\|x\|^{2}+B_{W}\|x\|\left\|y^{\prime}\right\|+C_{W}\left\|y^{\prime}\right\|^{2},
\end{equation*}

with

$A_{W}=A_{W_{1}} + A_{W_2}+ A_{W_3}=
\epsilon^{2} b_{4}  \ell_{1} \ell_{4} \left( \ell_{1} \ell_{4} + 2 \ell_{6} \right) + \epsilon^3 b_5 \ell_{6}^2 \left[ \ell_{1} \bar{r} +\ell_{2} (\bar{r}+\tilde{r}) \right] + \epsilon^2 b_4 \ell_{6}^2 = \epsilon^{2} w_{A_{1}} + \epsilon^{3} w_{A_{2}}$

$B_{W}=B_{W_{1}}+B_{W_{2}} + B_{W_3}=
\epsilon b_{4} \ell_{4}\left(2 \ell_{1} \ell_{3} + 2 \epsilon \ell_{1} \ell_{5} +2 \epsilon \ell_{1} \ell_{2} \ell_{4}  + 2 \epsilon \ell_{2} \ell_{6} \right) + \epsilon^2 b_{5} 2 \ell_{6} (\ell_{3} + \epsilon \ell_{5}) \left[ \ell_{1} \bar{r} +\ell_{2} (\bar{r}+\tilde{r}) \right] + \epsilon 2 b_4 \ell_{6} (\ell_{3} + \epsilon \ell_{5})
= \epsilon w_{B_{1}}+\epsilon^2 w_{B_{2}}+\epsilon^3 w_{B_{3}}$

$C_{W}=C_{W_{1}}+C_{W_{2}}+C_{W_{3}}-C_{W_{4}}=
\epsilon b_{4} \ell_{2}^{2}\left(2 \ell_{3} + 2 \epsilon \ell_{5} + \epsilon \ell_{2} \ell_{4}\right) +
\epsilon b_{5} (\ell_{3} + \epsilon \ell_{5})^{2} \left[ \ell_{1} \bar{r} +\ell_{2} (\bar{r}+\tilde{r}) \right]
+ \epsilon b_{4} \ell_{5} (2 \ell_{3} + \epsilon \ell_{5}) - b_{3}=-w_{C_{1}}+\epsilon w_{C_{2}}+\epsilon^{2} w_{C_{3}}$,

for suitable positive constants $w_{A_{1}}, w_{A_2}, w_{B_{1}}, w_{B_{2}}, w_{B_3}, w_{C_{1}}, w_{C_{2}}, w_{C_{3}}$.

Now it is possible to evaluate $\Delta U\left(x, y^{\prime}\right)=U\left(x^{+}, y^{\prime+}\right)-U\left(x, y^{\prime}\right)$. It holds

\begin{equation*}
\Delta U\left(x, y^{\prime}\right)=\Delta W\left(y^{\prime}, x\right)+\Delta V\left(y^{\prime}, x\right),
\end{equation*}

and the following bound easily follows

\begin{equation*}
\begin{aligned}
\Delta U\left(x, y^{\prime}\right) & \leq A_{V}\|x\|^{2}+B_{V}\|x\|\|y\|+C_{V}\left\|y^{\prime}\right\|^{2}+A_{W}\|x\|^{2}+B_{W}\|x\|\left\|y^{\prime}\right\|+C_{W}\left\|y^{\prime}\right\|^{2} \\
& \leq\left(A_{V}+A_{W}\right)\|x\|^{2}+\left(B_{V}+B_{W}\right)\|x\|\|y\|+\left(C_{V}+C_{W}\right)\|y\|^{2} .
\end{aligned}
\end{equation*}

This quadratic form can be rewritten as

\begin{equation*}
\Delta U\left(x, y^{\prime}\right) \leq\left[\begin{array}{ll}
\|x\| & \left\|y^{\prime}\right\|
\end{array}\right] \underbrace{\left[\begin{array}{cc}
A_{V}+A_{W} & \frac{1}{2}\left(B_{V}+B_{W}\right) \\
\frac{1}{2}\left(B_{V}+B_{W}\right) & C_{V}+C_{W}
\end{array}\right]}_{Q_{U}}\left[\begin{array}{l}
\|x\| \\
\left\|y^{\prime}\right\|
\end{array}\right] .
\end{equation*}

Now the aim is to show that $\Delta U\left(x, y^{\prime}\right)$ is always a negative quantity for a small enough choice of the parameter $\epsilon$. It is therefore necessary to study whether matrix $Q_{U}$ is negative definite:

\begin{equation*}
Q_{U}=\left[\begin{array}{cc}
-\epsilon v_{A_{1}}+\epsilon^{2} w_{A_{1}}  + \epsilon^{3} w_{A_{2}} & \frac{1}{2}\left(\epsilon v_{B_{1}}+\epsilon^{2} v_{B_{2}}+\epsilon w_{B_{1}}+\epsilon^{2} w_{B_{2}} + \epsilon^3 w_{B_3}\right) \\
\frac{1}{2}\left(\epsilon v_{B_{1}}+\epsilon^{2} v_{B_{2}}+\epsilon w_{B_{1}}+\epsilon^{2} w_{B_{2}} + \epsilon^3 w_{B_3}\right) & \epsilon^{2} v_{C_{1}}-w_{C_{1}}+\epsilon w_{C_{2}}+\epsilon^{2} w_{C_{3}}
\end{array}\right]
\end{equation*}

In order to verify whether $Q_{U}$ is negative definite, it is enough to verify whether the first principal minor is negative and the second one is positive for some choices of $\epsilon$. These quantities, in Landau notation, are respectively

\begin{equation*}
-\epsilon v_{A_{1}}+o(\epsilon), \quad \epsilon v_{A_{1}} w_{C_{1}}+o(\epsilon)
\end{equation*}

and so, there exists $\epsilon_{r}^{\prime}$ such that for $\epsilon<\epsilon_{r}^{\prime}$ matrix $Q_{U}$ is negative definite. As a consequence, there exists a positive constant $\ell_{U}$ such that

\begin{equation*}
Q_{U} \leq-\epsilon \ell_{U} I
\end{equation*}

and so

\begin{equation*}
\begin{aligned}
\Delta U\left(x, y^{\prime}\right) & \leq-\epsilon \ell_{U}\left(\|x\|^{2}+\left\|y^{\prime}\right\|^{2}\right) \\
& \leq-\epsilon \ell_{U}\left(\frac{1}{c_{1}} V(x)+\frac{1}{b_{1}} W\left(y^{\prime}, x\right)\right) \\
& \leq-\epsilon \gamma_{r} U\left(x, y^{\prime}\right)
\end{aligned}
\end{equation*}

with $\gamma_{r} \coloneqq \ell_{U} \max \left\{\frac{1}{c_{1}}, \frac{1}{b_{1}}\right\}$

From the latter inequality it follows, for some $\ell>0$,

\begin{equation*}
\|x(k)\|^{2}+\left\|y^{\prime}(k)\right\|^{2} \leq \ell\left(1-\epsilon \gamma_{r}\right)^{k}\left(\|x(0)\|^{2}+\left\|y^{\prime}(0)\right\|^{2}\right),
\end{equation*}

and since $\|x(0)\|^{2}+\left\|y^{\prime}(0)\right\| \leq r$, there exists a $C_{r}>0$ such that

\begin{equation*}
\|x(k)\|^{2} \leq C_{r}\left(1-\epsilon \gamma_{r}\right)^{k}, \quad \epsilon \in\left(0, \epsilon_{r}\right)
\end{equation*}

where $\epsilon_{r}=\min \left\{\epsilon_{c}, \epsilon_{r}^{\prime}\right\}$

\end{proof}


\section{Properties of the derivative estimators}
\label{appendix:estimators}

This appendix contains the derivation of the properties of the two derivative estimators, for a function $f(x): \mathbb{R}^d \to \mathbb{R}$ that satisfies Assumption \ref{assumpt:convexity_smoothness}. For convenience we rewrite here the expression of the estimators:
\begin{equation*}
\begin{aligned}
    \quad \hat{\nabla} f(\mu, x) &:= \sum_{k=1}^d \frac{f(x + \mu e_k)-f(x - \mu e_k)}{2 \mu} e_k
    \\
    \hat{\nabla}^2 f(\mu, x) &:= \sum_{k=1}^d \frac{f(x + \mu e_k)-2f(x)+f(x - \mu e_k)}{\mu^2} e_k
\end{aligned}
\end{equation*}
We also define
\begin{equation*}
\hat{\nabla} f(0,x) := \lim_{\mu \to 0} \hat{\nabla} f(\mu, x) = \nabla f(\mu,x)
\end{equation*}
and similarly for its derivatives. Given two matrices $A,B \in \mathbb{R}^{n \times n}$, we will often use the chains of inequalities
\begin{equation*}
\left[ A \right]_{ij} \leq \max_{i,j} \left[ A \right]_{ij} := \left\| A \right\|_{\text{max}} \leq \left\| A \right\| \leq n \left\| A \right\|_{\text{max}}
\end{equation*}
\begin{equation*}
\sigma_{\text{min}}(A) \left\| B \right\|
\leq \left\| AB \right\|
\leq \left\| A \right\| \left\| B \right\|
\end{equation*}
\begin{equation*}
\sigma_{\text{min}}(AB) \geq \sigma_{\text{min}}(A) \sigma_{\text{min}}(B)
\qquad
\sigma_{\text{max}}(AB) \leq \sigma_{\text{max}}(A) \sigma_{\text{max}}(B)
\end{equation*}
where $\sigma_{\text{min}}(\cdot)$ and $\sigma_{\text{max}}(\cdot)$ are the minimum and maximum singular values of the argument.

\subsection{Gradient estimator}
Using Taylor expansion with integral remainder on a generic element of the gradient estimator
\begin{equation*}
\begin{aligned}
\left[ \hat{\nabla} f(x) \right]_i &= \frac{f(x + \mu e_i) - f(x - \mu e_i)}{2 \mu} \\
&= \frac{1}{2 \mu} \left( f(x) + \nabla f(x)^T \mu e_i + \mu^2 \int_{0}^{1} (1-t) \left[ \nabla^2 f(x + t \mu e_i) \right]_{ii} \,dt \right) \\
& \quad - \frac{1}{2 \mu} \left( f(x) - \nabla f(x)^T \mu e_i + \mu^2 \int_{0}^{1} (1-t) \left[ \nabla^2 f(x - t \mu e_i) \right]_{ii} \,dt \right) \\
&= \left[ \nabla f(x) \right]_i + \underbrace{ \frac{\mu}{2} \int_{0}^{1} (1-t) \left( \left[ \nabla^2 f(x + t \mu e_i) \right]_{ii} - \left[ \nabla^2 f(x - t \mu e_i) \right]_{ii} \right) \,dt }_{:= h_i(\mu,x)}
\end{aligned}
\end{equation*}

from which we can see that the estimator is Lipschitz continuous:
\begin{equation*}
\begin{aligned}
\left\| \hat{\nabla} f(x) - \hat{\nabla} f(y) \right\| & \leq
\left\| \nabla f(x) - \nabla f(y) \right\| \\
& \quad + \frac{\mu}{2} \sqrt{ \sum_{i=1}^d \left( \int_{0}^{1} \left\lvert 1-t \right\rvert \left\lvert \left[ \nabla^2 f(x + t \mu e_i) \right]_{ii} - \left[ \nabla^2 f(y + t \mu e_i) \right]_{ii} \right\rvert \,dt \right)^2 } \\
& \quad + \frac{\mu}{2} \sqrt{ \sum_{i=1}^d \left( \int_{0}^{1} \left\lvert 1-t \right\rvert \left\lvert \left[ \nabla^2 f(x - t \mu e_i) \right]_{ii} - \left[ \nabla^2 f(y - t \mu e_i) \right]_{ii} \right\rvert \,dt \right)^2 } \\
& \leq L_1 \left\| x-y \right\| + \mu \sqrt{ \sum_{i=1}^d \left( \int_{0}^{1} \left\lvert 1-t \right\rvert L_2 \left\| x-y \right\| \,dt \right)^2 } \\
&= L_1 \left\| x-y \right\| + \frac{\mu \sqrt{d} L_2}{2} \left\| x-y \right\| \\
&= \left( L_1 + \frac{\mu \sqrt{d} L_2}{2} \right)\left\| x-y \right\|
\end{aligned}
\end{equation*}

\bigbreak \noindent
We can write
\begin{equation*}
\hat{\nabla} f(x) = \nabla f(x) + h(\mu,x)
\end{equation*}
where the approximation error is bounded by
\begin{equation*}
\begin{aligned}
\left\| h(\mu,x) \right\| & \leq
\frac{\mu}{2} \sqrt{ \sum_{i=1}^d \left( \int_{0}^{1} |1-t| \left\lvert \left[ \nabla^2 f(x + t \mu e_i) \right]_{ii} - \left[ \nabla^2 f(y + t \mu e_i) \right]_{ii} \right\rvert \,dt \right)^2 } \\
& \leq \frac{\mu}{2} \sqrt{ \sum_{i=1}^d \left( \int_{0}^{1} |1-t| L_2 \left\| 2t \mu e_i \right\| \,dt \right)^2 } \\
& = \frac{\mu}{2} \sqrt{ \sum_{i=1}^d 2 L_2 \mu \left( \int_{0}^{1} t|1-t| \,dt \right)^2 } \\
&= \frac{\sqrt{d} L_2 \mu^2}{6}
\end{aligned}
\end{equation*}

\subsection{Hessian estimator}
Using Taylor expansion with integral remainder on a generic element of the Hessian estimator
\begin{equation*}
\begin{aligned}
\left[ \hat{\nabla}^2 f(x) \right]_i &= \frac{f(x + \mu e_i) - 2f(x) + f(x - \mu e_i)}{\mu^2} \\
&= \left[ \nabla^2 f(x) \right]_{ii} + \frac{\mu}{2} \int_{0}^{1} (1-t)^2 \left( \left[ \left( D^3 f \right)(x + t \mu e_i) \right]_{iii} - \left[ \left( D^3 f \right)(x - t \mu e_i) \right]_{iii} \right) \,dt
\end{aligned}
\end{equation*}

from which we can see that the estimator is Lipschitz continuous:
\begin{equation*}
\begin{aligned}
\left\| \hat{\nabla}^2 f(x) - \hat{\nabla}^2 f(y) \right\| & \leq
\sqrt{ \sum_{i=1}^d \left( \left[ \nabla^2 f(x) \right]_{ii} - \left[ \nabla^2 f(y) \right]_{ii} \right)^2 } \\
& \quad + \frac{\mu}{2} \sqrt{ \sum_{i=1}^d \left( \int_{0}^{1} (1-t)^2 \left\lvert \left[ \left( D^3 f \right)(x + t \mu e_i) \right]_{iii} - \left[ \left( D^3 f \right)(y + t \mu e_i) \right]_{iii} \right\rvert \,dt  \right)^2 } \\
& \quad + \frac{\mu}{2} \sqrt{ \sum_{i=1}^d \left( \int_{0}^{1} (1-t)^2 \left\lvert \left[ \left( D^3 f \right)(x - t \mu e_i) \right]_{iii} - \left[ \left( D^3 f \right)(y - t \mu e_i) \right]_{iii} \right\rvert \,dt \right)^2 } \\
& \leq L_2 \sqrt{d} \left\| x-y \right\| + \mu \sqrt{ \sum_{i=1}^d \left( \int_{0}^{1} (1-t)^2  L_3 \left\| x-y \right\| \,dt \right)^2 } \\
& \leq L_2 \sqrt{d} \left\| x-y \right\| + \frac{\mu \sqrt{d} L_3}{3} \left\| x-y \right\| \\
& = \left( L_2 \sqrt{d} + \frac{\mu \sqrt{d} L_3}{3} \right) \left\| x-y \right\|
\end{aligned}
\end{equation*}

\bigbreak \noindent
The approximation error is bounded by
\begin{equation*}
\begin{aligned}
\left\lvert \left[ \hat{\nabla}^2 f(x) \right]_{i} - \left[ \nabla^2 f(x)\right]_{ii} \right\rvert &=
\left\lvert \frac{\mu}{2} \int_{0}^{1} (1-t)^2 \left( \left[ \left( D^3 f \right)(x + t \mu e_i) \right]_{iii} - \left[ \left( D^3 f \right)(x - t \mu e_i) \right]_{iii} \right) \,dt \right\rvert \\
& \leq \frac{\mu}{2} \int_{0}^{1} (1-t)^2 \left\lvert \left[ \left( D^3 f \right)(x + t \mu e_i) \right]_{iii} - \left[ \left( D^3 f \right)(x - t \mu e_i) \right]_{iii} \right\rvert \,dt \\
& \leq \frac{\mu}{2} \int_{0}^{1} (1-t)^2 L_3 \left\| 2 t \mu e_i \right\| \,dt \\
& \leq \frac{L_3 \mu^2}{12} \quad i=1, \dots, d.
\end{aligned}
\end{equation*}

\bigbreak \noindent
We can now prove Lemma \ref{lemma:estimators_exact_quadratic_case}:
\begin{proof}
    Consider a generic quadratic function
    \begin{equation*}
        f(x) = \frac{1}{2} x^T A x + b^T x + c
        \qquad
        \nabla^2 f(x) = A
        \qquad
        \left[ \left( D^3 f \right)(x) \right]_{ijk} = 0 \quad \forall i,j,k
    \end{equation*}
    Then 
    \begin{equation*}
    \begin{aligned}
    \left[ \hat{\nabla} f(x) \right]_i &=
    \left[ \nabla f(x) \right]_i + \frac{\mu}{2} \int_{0}^{1} (1-t) \left( \left[ \nabla^2 f(x + t \mu e_i) \right]_{ii} - \left[ \nabla^2 f(x - t \mu e_i) \right]_{ii} \right) \,dt \\
    &= \left[ \nabla f(x) \right]_i + \frac{\mu}{2} \int_{0}^{1} (1-t) \left( \left[ A \right]_{ii} - \left[ A \right]_{ii} \right) \,dt
    = \left[ \nabla f(x) \right]_i
    \end{aligned}
    \end{equation*}
    \begin{equation*}
    \begin{aligned}
    \left[ \hat{\nabla}^2 f(x) \right]_i &=
    \left[ \nabla^2 f(x) \right]_{ii} + \frac{\mu}{2} \int_{0}^{1} (1-t)^2 \left( \left[ \left( D^3 f \right)(x + t \mu e_i) \right]_{iii} - \left[ \left( D^3 f \right)(x - t \mu e_i) \right]_{iii} \right) \,dt \\
    &= \left[ \nabla^2 f(x) \right]_{ii} + \frac{\mu}{2} \int_{0}^{1} (1-t)^2 \left( 0-0 \right) \,dt
    = \left[ \nabla^2 f(x) \right]_{ii}
    \end{aligned}
\end{equation*}
\end{proof}

\subsection{Jacobian of the gradient estimator}
A generic element of the Jacobian satisfies
\begin{equation*}
\begin{aligned}
\left[ \frac{\partial \hat{\nabla} f(x)}{\partial x} \right]_{ij}
& = \frac{ \nabla f(x + \mu e_i)^T e_j - \nabla f(x - \mu e_i)^T e_j }{2 \mu} \\
& = \frac{ \left[ \nabla f(x) \right]_j +  \mu \left[ \nabla^2 f(x) \right]_{ij} + \mu^2 \int_{0}^{1} (1-t) \left[ \left( D^3 f \right)(x + t \mu e_i) \right]_{iij} \,dt }{2 \mu} \\
& \quad - \frac{ \left[ \nabla f(x) \right]_j -  \mu \left[ \nabla^2 f(x) \right]_{ij} + \mu^2 \int_{0}^{1} (1-t) \left[ \left( D^3 f \right)(x - t \mu e_i) \right]_{iij} \,dt }{2 \mu} \\
& = \left[ \nabla^2 f(x) \right]_{ij} + \frac{\mu}{2} \int_{0}^{1} (1-t) \left( \left[ \left( D^3 f \right)(x + t \mu e_i) \right]_{iij} - \left[ \left( D^3 f \right)(x - t \mu e_i) \right]_{iij} \right) \,dt .
\end{aligned}
\end{equation*}
Define $R(\mu, x)$ as the $d \times d$ matrix which contains the integral terms above, so that
\begin{equation*}
\frac{\partial \hat{\nabla} f(x)}{\partial x} = \nabla^2 f(x) + R(\mu, x)
\end{equation*}
\begin{equation*}
\begin{aligned}
\left\lvert \left[ R(\mu, x) \right]_{ij} \right\rvert
&= \left\lvert \frac{\mu}{2} \int_{0}^{1} (1-t) \left( \left[ \left( D^3 f \right)(x + t \mu e_i) \right]_{iij} - \left[ \left( D^3 f \right)(x - t \mu e_i) \right]_{iij} \right) \,dt \right\rvert \\
& \leq \frac{\mu}{2} \int_{0}^{1} (1-t) L_3 \left\| 2t \mu e_i \right\| \,dt \\
&= \frac{\mu^2 L_3}{6}  \quad \forall i,j \quad \forall x
\end{aligned}
\end{equation*}

\subsection{Zeros of the gradient estimator}
Here we provide a proof of Lemma \ref{lemma:implicit_function}.

\begin{proof}
Let $x^\star$ be the unique global minimizer of the strongly convex function $f(x)$, for which $\nabla f(x^\star) = 0$. Given $\mu$, we want to characterize the set of solutions to the equation $\hat{\nabla}f(\mu, x) = 0$. Note that $\hat{\nabla}f(\mu, x)$ is twice differentiable in both its arguments $x$ and $\mu \neq 0$. Recall that
\begin{equation*}
\hat{\nabla} f(0,x) := \lim_{\mu \to 0} \hat{\nabla} f(\mu, x) = \nabla f(\mu,x)
\end{equation*}
and similarly for its derivatives, so that $\hat{\nabla} f(0,x^\star) = 0$.
The Jacobian matrix of $\hat{\nabla}f(\mu, x)$ is $ \left( D \hat{\nabla}f \right) \left( \mu, x \right) = \left[ \left( D_\mu \hat{\nabla}f \right) \left( \mu, x \right) , \left( D_x \hat{\nabla}f \right) \left( \mu, x \right) \right] :\ \mathbb{R}^{d+1} \to \mathbb{R}^{d \times d+1} $, where
\begin{equation*}
\left( D_\mu \hat{\nabla}f \right) \left( \mu, x \right) =
\begin{bmatrix}
\dfrac{\partial \left[ \hat{\nabla}f \right]_1}{\partial \mu} \\
\vdots \\
\dfrac{\partial \left[ \hat{\nabla}f \right]_d}{\partial \mu}
\end{bmatrix}_{\mu, x}
\qquad
\left( D_x \hat{\nabla}f \right) \left( \mu, x \right) =
\begin{bmatrix}
\dfrac{\partial \left[ \hat{\nabla}f \right]_1}{\partial x_1}  & \cdots & \dfrac{\partial \left[ \hat{\nabla}f \right]_1}{\partial x_d}  \\
\vdots & \ddots & \vdots \\
\dfrac{\partial \left[ \hat{\nabla}f \right]_d}{\partial x_1}  & \cdots & \dfrac{\partial \left[ \hat{\nabla}f \right]_d}{\partial x_d}
\end{bmatrix}_{\mu, x}
\end{equation*}
We want to apply the Implicit Function theorem, and to do so we need $\left( D_x \hat{\nabla}f \right) \left( 0, x^\star \right)$ to be invertible. Indeed,
\begin{equation*}
\begin{aligned}
\left[ \left( D_x \hat{\nabla}f \right) \left( 0, x^\star \right) \right]_{ij}
& = \lim_{\mu \to 0,\ x \to x^\star} \left[ \left( D_x \hat{\nabla}f \right) \left( \mu, x \right) \right]_{ij}
= \lim_{\mu \to 0,\ x \to x^\star} \left[ \dfrac{\partial \left[ \hat{\nabla}f \right]_i}{\partial x_j} \right]_{ij} \\
& = \lim_{\mu \to 0,\ x \to x^\star} \frac{1}{2 \mu} \left( \dfrac{\partial f(x + \mu e_i)}{\partial x_j} - \dfrac{\partial f(x - \mu e_i)}{\partial x_j} \right) \\
& = \dfrac{\partial^2 f(x^\star)}{\partial x_i \partial x_j}
= \left[ \nabla^2 f(x^\star) \right]_{ij}
\end{aligned}
\end{equation*}
and therefore $\left( D_x \hat{\nabla}f \right) \left( 0, x^\star \right) = \nabla^2 f(x^\star) > 0$ is invertible. We also note that
\begin{equation*}
\begin{aligned}
\left[ \left( D_\mu \hat{\nabla}f \right) \left( 0, \Gamma(0) \right) \right]_i
&= \left. \dfrac{\partial \left[ \hat{\nabla}f \right]_1}{\partial \mu} \right\rvert_{0,x^\star}
= \left. \frac{\partial}{\partial \mu} \left[ \frac{f(x + \mu e_i) - f(x - \mu e_i)}{2 \mu} \right] \right\rvert_{0,x^\star} \\
&= \left. \frac{\partial}{\partial \mu} \left[ \frac{\mu}{2} \int_{0}^{1} (1-t) \left( \left[ \nabla^2 f(x + t \mu e_i) \right]_{ii} - \left[ \nabla^2 f(x - t \mu e_i) \right]_{ii} \right) \,dt \right] \right\rvert_{0,x^\star} \\
&= \left. \left[ \frac{1}{2} \int_{0}^{1} (1-t) \left( \left[ \nabla^2 f(x + t \mu e_i) \right]_{ii} - \left[ \nabla^2 f(x - t \mu e_i) \right]_{ii} \right) \,dt \right] \right\rvert_{0,x^\star} \\
& \quad \left. + \left[ \frac{\mu}{2} \frac{\partial}{\partial \mu}  \int_{0}^{1} (1-t) \left( \left[ \nabla^2 f(x + t \mu e_i) \right]_{ii} - \left[ \nabla^2 f(x - t \mu e_i) \right]_{ii} \right) \,dt \right] \right\rvert_{0,x^\star} \\
&=  0 + \left. \left[ \frac{\mu}{2} \int_{0}^{1} (1-t) \left( \left[ \frac{\partial}{\partial \mu} \nabla^2 f(x + t \mu e_i) \right]_{ii} - \left[ \frac{\partial}{\partial \mu} \nabla^2 f(x - t \mu e_i) \right]_{ii} \right) \,dt \right] \right\rvert_{0,x^\star} = 0
\end{aligned}
\end{equation*}
where we applied Leibniz integral rule and we used the fact that by Lipschitz continuity of the Hessian
\begin{equation*}
\begin{aligned}
\left\lvert \int_{0}^{1} (1-t) \left[ \frac{\partial}{\partial \mu} \nabla^2 f(x + t \mu e_i) \right]_{ii} \,dt \right\rvert
& \leq \int_{0}^{1} |1-t| \left\lvert \left[ \frac{\partial}{\partial \mu} \nabla^2 f(x + t \mu e_i) \right]_{ii} \right\rvert \,dt \\
& = \int_{0}^{1} |1-t| \lim_{\Delta \mu \to 0} \left[ \left\lvert \nabla^2 f(x + t \mu e_i + t \Delta \mu e_i) - \nabla^2 f(x + t \mu e_i) \right\rvert \right]_{ii} \,dt \\
& \leq \int_{0}^{1} |1-t| \lim_{\Delta \mu \to 0} L_2 \left\| t \Delta \mu e_i \right\| \,dt = 0
\end{aligned}
\end{equation*}
By the Implicit Function theorem, there exist $\bar{\mu}, \delta' > 0$ and a continuously differentiable function $\Gamma:\ \mathbb{R} \to \mathbb{R}^d$ such that for all $|\mu - 0| = \mu \leq \bar{\mu}$ there is a unique $x = \Gamma(\mu) \in \mathbb{R}^d$ such that $\left\| \Gamma(\mu) - x^\star \right\| \leq \delta'$ and $\hat{\nabla} f(\mu, \Gamma(\mu)) = 0$.

\bigbreak \noindent
We are now left to show that $\left\| \Gamma(\mu) - x^\star \right\| \leq \mu^2 \delta$ for some positive $\delta$. Using Taylor expansion we can write
\begin{equation*}
\begin{aligned}
\left\| \Gamma(\mu) - x^\star \right\|
&= \left\| \Gamma(0) + \mu \left( D \Gamma \right)(0) + \mu^2 \int_{0}^{1} (1-t) \left( D^2 \Gamma \right)(\mu) \,dt - x^\star \right\| \\
&= \mu^2 \left\| \int_{0}^{1} (1-t) \left( D \Gamma \right)(\mu) \,dt \right\|
\leq \mu^2 \delta
\end{aligned}
\end{equation*}
where we used the fact that $\Gamma(0) = x^\star$ and $\left( D \Gamma \right)(0) = 0$. Indeed, by the Implicit Function theorem, for all $ \mu \leq \bar{\mu}$ the Jacobian of $\Gamma(\mu)$ is given by
\begin{equation*}
\left( D \Gamma \right)(\mu) = - \left[ \left( D_x \hat{\nabla}f \right) \left( \mu, \Gamma(\mu) \right) \right]^{-1} \left( D_\mu \hat{\nabla}f \right) \left( \mu, \Gamma(\mu) \right)
\end{equation*}
and we already proved that $\left( D_\mu \hat{\nabla}f \right) \left( 0, \Gamma(0) \right) = 0$.

\end{proof}

\subsection{Square norm of the gradient estimator}
In the proof of Theorem \ref{theorem:convergence_ZO_JADE} we use the square norm of the gradient estimator as a Lyapunov function, and to do so we need to analyze some properties of this quantity. In the following we use the notation $\Gamma_\mu = \Gamma(\mu)$.

\subsubsection{Upper bound}
Using Lipschitz continuity, the norm of the gradient estimator can be bounded by
\begin{equation*}
\left\| \hat{\nabla} f(x) \right\|^2
= \left\| \hat{\nabla} f(x) - \hat{\nabla} f(\Gamma_\mu) \right\|^2
\leq \left( L_1 + \frac{\mu \sqrt{d} L_2}{2} \right)^2 \left\| x-\Gamma_\mu \right\|^2
\end{equation*}

\subsubsection{Lower bound}
We consider the Taylor expansions
\begin{equation*}
\left\| \hat{\nabla} f(x) \right\|^2
= \left\| \hat{\nabla} f(\Gamma_\mu) \right\|^2 + (x - \Gamma_\mu)^T \int_0^1 \frac{\partial \left\| \hat{\nabla} f \left( \Gamma_\mu + t(x-\Gamma_\mu) \right) \right\|^2 }{\partial x} \, dt
\end{equation*}
\begin{equation*}
\hat{\nabla} f\left( \Gamma_\mu + t(x-\Gamma_\mu) \right) =  \hat{\nabla} f(\Gamma_\mu) + \int_0^1 \frac{\partial \hat{\nabla} f \left( \Gamma_\mu + pt(x-\Gamma_\mu) \right)}{\partial x} \, dp \cdot t(x - \Gamma_\mu) .
\end{equation*}
Recalling that $\hat{\nabla} f(\Gamma_\mu) = 0$ and $\frac{\partial \hat{\nabla} f(x)}{\partial x} = \nabla^2 f(x) + R(\mu, x)$, and using the notation
\begin{equation*}
\frac{\partial \hat{\nabla} f \left( \Gamma_\mu + t(x-\Gamma_\mu) \right) }{\partial x} = \nabla_t^2 + R_t
\qquad
\frac{\partial \hat{\nabla} f \left( \Gamma_\mu + pt(x-\Gamma_\mu) \right) }{\partial x} = \nabla_{pt}^2 + R_{pt}
\end{equation*}
we can write
\begin{equation*}
\begin{aligned}
\frac{\partial \left\| \hat{\nabla} f \left( \Gamma_\mu + t(x-\Gamma_\mu) \right) \right\|^2 }{\partial x}
&= 2 \frac{\partial \hat{\nabla} f \left( \Gamma_\mu + t(x-\Gamma_\mu) \right) }{\partial x}^T \hat{\nabla} f \left( \Gamma_\mu + t(x-\Gamma_\mu) \right) \\
&= 2 \frac{\partial \hat{\nabla} f \left( \Gamma_\mu + t(x-\Gamma_\mu) \right) }{\partial x}^T \int_0^1 \frac{\partial \hat{\nabla} f \left( \Gamma_\mu + pt(x-\Gamma_\mu) \right)}{\partial x} \, dp \cdot t(x - \Gamma_\mu) \\
&= 2 \int_0^1 \frac{\partial \hat{\nabla} f \left( \Gamma_\mu + t(x-\Gamma_\mu) \right) }{\partial x}^T \frac{\partial \hat{\nabla} f \left( \Gamma_\mu + pt(x-\Gamma_\mu) \right)}{\partial x} \, dp \cdot t(x - \Gamma_\mu) \\
&= 2 \int_0^1 \left( \nabla_t^2 + R_t \right)^T \left( \nabla_{pt}^2 + R_{pt} \right) \, dp \cdot t(x - \Gamma_\mu) \\
\end{aligned}
\end{equation*}
Plugging this result into the quantity of interest and using
\begin{equation*}
\sigma_{\text{max}}(R) =
\sqrt{ \lambda_{\text{max}}(R^T R) } =
\sqrt{ \left\| R^T R \right\| } \leq
\sqrt{ d \left\| R^T R \right\|_{\text{max}} } \leq
\sqrt{ d^2 \left( \frac{\mu^2 L_3}{6} \right)^2 } =
\frac{d \mu^2 L_3}{6}
\end{equation*}
we get
\begin{equation*}
\begin{aligned}
\left\| \hat{\nabla} f(x) \right\|^2
&= (x - \Gamma_\mu)^T \int_0^1 2t \int_0^1 \nabla_t^2 \nabla_{pt}^2 + \nabla_t^2 R_{pt} + R_t^T \nabla_{pt}^2 + R_t^T R_{pt} \, dp \, dt \cdot (x - \Gamma_\mu)\\
& \geq \left\| x - \Gamma_\mu \right\|^2 \left( m^2 - 2 L_1 \frac{d \mu^2 L_3}{6} -\left( \frac{d \mu^2 L_3}{6} \right)^2 \right)
\end{aligned}
\end{equation*}
which for $L_3 \neq 0$ is a positive number provided that
\begin{equation*}
\mu < \sqrt{ \frac{6 \left( \sqrt{L_1^2 + m^2} - L_1 \right) }{d L_3} } .
\end{equation*}

\subsubsection{Norm of the partial derivative}
We want to obtain an upper bound on 
\begin{equation*}
\begin{aligned}
\left\| \frac{\partial \left\| \hat{\nabla} f(x) \right\|^2 }{\partial x} \right\| &=
\left\| 2 \hat{\nabla} f(x)^T \frac{\partial \hat{\nabla} f(x)}{\partial x} \right\| \\
&\leq \left\| 2 \hat{\nabla} f(x)^T \nabla^2 f(x) \right\| + \left\| 2 \hat{\nabla} f(x)^T R(\mu, x) \right\| \\
& \leq 2 L_1 \left\| \hat{\nabla} f(x) \right\| + \frac{\mu^2 L_3 d}{3} \left\| \hat{\nabla} f(x) \right\| \\
& \leq \left(2 L_1 + \frac{\mu L_3 d}{3} \right) \left( L_1 + \frac{\mu \sqrt{d} L_2}{2} \right) \left\| x-\Gamma_\mu \right\|
\end{aligned}
\end{equation*}
Where we used 
\begin{equation*}
\hat{\nabla} f(x)^T R(\mu, x) = \hat{\nabla} f(x)^T [R_1 \cdots R_d]
\qquad
\hat{\nabla} f(x)^T R_i
\leq \left\| \hat{\nabla} f(x) \right\| \left\| R_i \right\|
\leq \frac{\mu L_3 \sqrt{d}}{6} \left\| \hat{\nabla} f(x) \right\|
\end{equation*}

\subsubsection{Descent direction}
We want to prove that
\begin{equation*}
2 \hat{\nabla} f(x)^T \frac{\partial \hat{\nabla} f(x)}{\partial x} \left( - \hat{\nabla} f(x) \oslash \hat{\nabla}^2 f(x) \right)
\leq \alpha \left\| x - x^\star \right\|^2
\end{equation*}
for some negative scalar $\alpha$. We can rewrite the quantity of interest as
\begin{equation*}
\begin{aligned}
& - 2 \hat{\nabla} f(x)^T \left[ \frac{\partial \hat{\nabla} f(x)}{\partial x} \text{ diag}\left( \mathbb{1} \oslash \hat{\nabla}^2 f(x) \right) \right] \hat{\nabla} f(x) \\
& = - 2 \hat{\nabla} f(x)^T \left[ \nabla^2 f(x) \text{ diag}\left( \mathbb{1} \oslash \hat{\nabla}^2 f(x) \right) \right] \hat{\nabla} f(x)
- 2 \hat{\nabla} f(x)^T \left[ R(\mu, x) \text{ diag}\left( \mathbb{1} \oslash \hat{\nabla}^2 f(x) \right) \right] \hat{\nabla} f(x) \\
& \leq - \left\| \hat{\nabla} f(x) \right\|^2 \sigma_{\text{min}} \left[ \nabla^2 f(x) \text{ diag}\left( \mathbb{1} \oslash \hat{\nabla}^2 f(x) \right) \right]
+ \left\| \hat{\nabla} f(x) \right\|^2 \sigma_{\text{max}} \left[ R(\mu, x) \text{ diag}\left( \mathbb{1} \oslash \hat{\nabla}^2 f(x) \right) \right] \\
& \leq - \left\| \hat{\nabla} f(x) \right\|^2 \sigma_{\text{min}} \left[ \nabla^2 f(x) \right] \sigma_{\text{min}} \left[ \text{ diag}\left( \mathbb{1} \oslash \hat{\nabla}^2 f(x) \right) \right]
+ \left\| \hat{\nabla} f(x) \right\|^2 \sigma_{\text{max}}\left[ R(\mu, x) \right] \sigma_{\text{max}} \left[ \text{ diag}\left( \mathbb{1} \oslash \hat{\nabla}^2 f(x) \right) \right] \\
& \leq \left\| \hat{\nabla} f(x) \right\|^2 \left( -m \dfrac{1}{L_1 + \frac{L_3 \mu^2}{12} } + \frac{d \mu^2 L_3}{6} \frac{1}{m - \frac{L_3 \mu^2}{12} } \right) \\
& \leq \left( \frac{2 d \mu^2 L_3}{2m - L_3 \mu^2} -\frac{12m}{12 L_1 + L_3 \mu^2}  \right) \left( L_1 + \frac{\mu \sqrt{d} L_2}{2} \right)^2 \left\| x -\Gamma_\mu \right\|^2
\end{aligned}
\end{equation*}
and when $L_3 \neq 0$ the coefficient is negative provided that 
\begin{equation*}
\mu < \sqrt{ 3 \frac{ \sqrt{4d^2 L_1^2 + m^2 + 4dL_1 m + 8 m^2 d} - 2d L_1 - m }{d L_3} } .
\end{equation*}

\bibliographystyle{IEEEtran}
\bibliography{IEEEfull, references}

\end{document}